\documentclass[10pt,a4paper, ]{article}
\usepackage[toc,page]{appendix}
\usepackage[utf8]{inputenc}
\usepackage{amsmath}
\usepackage{amsfonts}
\usepackage{mathtools}
\usepackage{graphicx}
\usepackage{cite}
\usepackage{amsthm}
\usepackage[yyyymmdd,hhmmss]{datetime}
\usepackage{stmaryrd}
\newtheorem{theorem}{Theorem}[section]
\newtheorem{lemma}[theorem]{Lemma}
\newtheorem{remark}[theorem]{Remark}

\newtheorem{definition}[theorem]{Definition}
\usepackage{amssymb}
\newcommand{\R}{\mathbb{R}  }
\newcommand{\Rn}{ {\mathbb{R}^n}}
\newcommand{\rnplus}{{\mathbb{R}^n_+ } }
\newcommand{\rn}{{\mathbb{R}^n } }

\newcommand{\pfeil}{ \rightarrow }

\newcommand{\into}{\hookrightarrow}

\renewcommand{\phi}{\varphi}
\renewcommand{\div}{\operatorname{div}}
\newcommand{\ljump}{ \llbracket  }
\newcommand{\rjump}{ \rrbracket }

\usepackage{authblk}
\title{Well-Posedness and qualitative behaviour of the Mullins-Sekerka problem with ninety-degree angle boundary contact}
\author[1]{Helmut Abels}
\author[1]{Maximilian Rauchecker}
\author[2]{Mathias Wilke}

\affil[1]{Fakultät für Mathematik, Universität Regensburg,  93053 Regensburg, Germany}
\affil[2]{Institut für Mathematik, Martin-Luther-Universität Halle-Wittenberg, 06099 Halle, Germany}
\numberwithin{equation}{section} 

\begin{document}
\maketitle
\begin{abstract}
We show local well-posedness for a Mullins-Sekerka system with ninety degree angle boundary contact. We will describe the motion of the moving interface by a height function over a fixed reference surface. Using the theory of maximal regularity together with a linearization of the equations and a localization argument we will prove well-posedness of the full nonlinear problem via the contraction mapping principle. Here one difficulty lies in choosing the right space for the Neumann trace of the height function and showing maximal $L_p-L_q$-regularity for the linear problem.

In the second part we show that solutions starting close to certain equilibria exist globally in time, are stable, and converge to an equilibrium solution at an exponential rate.
\end{abstract}
\section{Introduction}
In this article we study the Mullins-Sekerka problem inside a bounded, smooth domain $\Omega \subset \R^n, n = 2,3$, where the interface separating the two materials meets the boundary of $\Omega$ at a constant ninety degree angle. This leads to a free boundary problem involving a contact angle problem as well.

We assume that the domain $\Omega$ can be decomposed  as $\Omega = \Omega^+ (t) \dot\cup \mathring \Gamma (t) \dot \cup \Omega^- (t)$, where $\mathring \Gamma (t)$ denotes the interior of $\Gamma(t)$, a $(n-1)$-dimensional submanifold with boundary. We interpret $\Gamma(t)$ to be the interface separating the two phases, $\Omega^\pm (t)$, which will be assumed to be connected. The boundary of $\Gamma(t)$ will be denoted by $\partial \Gamma (t)$. Furthermore we assume $\Gamma(t)$ to be orientable, the unit vector field on $\Gamma(t)$ pointing from $\Omega^+(t)$ to $\Omega^-(t)$ will be denoted by $n_{\Gamma(t)}$. 

 The precise model we study reads as
\begin{subequations} \label{87959649764dfs}
\begin{align}
V_{\Gamma(t)} &= - \ljump n_{\Gamma(t)} \cdot \nabla \mu \rjump, && \text{on } \Gamma(t), \label{9283486} \\
\mu|_{\Gamma(t)} &= H_{\Gamma(t)}, && \text{on } \Gamma(t), \\
\Delta \mu &= 0, && \text{in } \Omega \backslash \Gamma(t), \\
n_{\partial\Omega} \cdot \nabla \mu|_{\partial\Omega} &= 0, && \text{on } \partial\Omega, \\
\mathring \Gamma(t) &\subseteq \Omega, \\
\partial \Gamma(t) &\subseteq \partial\Omega, \\
\angle(\Gamma(t), \partial \Omega) &= \pi/2, && \text{on } \partial\Gamma(t), \label{983457034}
\end{align}
subject to the initial condition
\begin{equation}
\Gamma|_{t=0} = \Gamma_0.\label{9283486Z}
\end{equation}
\end{subequations}
Here, $V_{\Gamma(t)}$ denotes the normal velocity and $H_{\Gamma(t)}$ the mean curvature of the free interface $\Gamma(t)$, which is given by the sum of the principal curvatures. 
By $\ljump \cdot \rjump$ we denote the jump of a quantity across $\Gamma(t)$ in direction of $n_{\Gamma(t)}$, that is,
\begin{equation}
\ljump f \rjump (x) := \lim_{ \varepsilon \pfeil 0+} [ f(x + \varepsilon n_{\Gamma(t)} ) -  f(x - \varepsilon n_{\Gamma(t)} ) ], \quad x \in \Gamma(t).
\end{equation}
Equation \eqref{983457034} prescribes the angle at which the interface $\Gamma(t)$ has contact with the fixed boundary $\partial\Omega$, which will be a constant ninety degree angle during the evolution. We can alternatively write \eqref{983457034} as the condition that the normals are perpendicular on the boundary of the interface,
\begin{equation}
n_{ \Gamma (t)} \cdot n_{\partial \Omega} = 0, \quad \text{on } \partial \Gamma (t).
\end{equation}

Let us first state some simple properties of this evolution. 
Note that we obtain the compatibility condition
\begin{equation}
\angle(\Gamma_0, \partial \Omega) = \pi/2 \quad \text{on } \partial\Gamma_0. 
\end{equation}
Furthermore, the volume of each of the two phases is conserved,
\begin{equation} \label{9348765037465087345}
\frac{d}{dt}| \Omega^\pm (t) | = 0, \quad t \in \R_+.
\end{equation}
Here, $\Omega^\pm (t)$ denote the two different phases separated by the sharp interface, $\Omega = \Omega^+(t) \cup \mathring \Gamma(t) \cup \Omega^-(t)$. Then \eqref{9348765037465087345} stems from
\begin{align}
\frac{d}{dt} | \Omega^+ (t)| = \int_{\Gamma(t)} V_{\Gamma(t)} d\mathcal H^{n-1}  &= - \int_{\Gamma(t)} \ljump n_{\Gamma(t)} \cdot \nabla \mu \rjump d\mathcal H^{n-1}  \\ &= \int_{\Omega^+(t)} \Delta \mu dx = 0.
\end{align}
However, the energy given by the surface area of the free interface $\Gamma(t)$ satisfies
\begin{equation}
\frac{d}{dt}| \Gamma(t) | \leq 0, \quad t \in \R_+.
\end{equation}
Indeed, an integration by parts readily gives
\begin{align}
\frac{d}{dt}|\Gamma(t)| = \int_{\Gamma(t)} H_{\Gamma(t)} V_{\Gamma(t)} d\mathcal H^{n-1} &= - \int_{\Gamma(t)} \mu|_{\Gamma(t)}  \ljump n_{\Gamma(t)} \cdot \nabla \mu \rjump d\mathcal H^{n-1} \\ &= - \int_\Omega | \nabla \mu|^2 dx \leq 0.
\end{align}

In this article we are concerned with existence of strong solutions of the Mullins-Sekerka problem \eqref{87959649764dfs}. To this end we will later pick some reference surface $\Sigma$ inside the domain $\Omega$, also intersecting the boundary with a constant ninety degree angle, and write the moving interface as a graph over $\Sigma$ by a height function $h$, depending on $x \in \Sigma$ and time $t \geq 0$. Pulling back the equations then to the time-independent domain $\Omega \backslash \Sigma$ we reduce the problem to a nonlinear evolution equation for $h$. The corresponding linearization for the spatial differential operator for $h$ then turns out to be a nonlocal pseudo-differential operator of order three, cf. \cite{eschersimonett}. We also refer to the introduction of Escher and Simonett \cite{eschersimonett} for further properties of the Mullins-Sekerka problem.

In the following, we will be interested in height functions $h$ with regularity
\begin{equation}
h \in W^1_p(0,T;W^{1-1/q}_q(\Sigma)) \cap L_p(0,T;W^{4-1/q}_q(\Sigma)),
\end{equation}
where $p$ and $q$ are different in general. We will choose $q < 2$ and $p$ finite but large, to ensure that the real interpolation space
\begin{equation}
X_\gamma := ( W^{4-1/q}_q(\Sigma), W^{1-1/q}_q(\Sigma) )_{1-1/p,p} = B^{4-1/q-3/p}_{qp}(\Sigma)
\end{equation}
continuously embeds into $C^2(\Sigma)$, cf. Amann \cite{amannlineartheory}.
By an ansatz where $p = q < 2$, this is not achievable. 
We need however the restriction $q < 2$ to avoid additional compatibility conditions for the elliptic problem, cf. also Section \ref{section34345439}.
This however requires an $L_p-L_q$ maximal regularity result of the underlying linearized problem, which we will also show in this article.

\textbf{Outline of this paper.} In Section 2 we will briefly introduce function spaces and techniques we work with and give references for further discussion. In section 3 we rewrite the free boundary problem of the moving interface as a nonlinear problem for the height function parametrizing the interface. Section 4 is devoted to the analysis of the underlying linear problem, where an extensive analysis is made on the half-space model problems.
This is needed since these model problems at the contact line are not well-understood until now.
 The main result of this section is $L_p-L_q$ maximal regularity for the linear problem. Section 5 contains that the full nonlinear problem is well-posed and Section 6 is concerned with the stability properties of solutions starting close to certain equilibria.
\section{Preliminaries and Function Spaces}
In this section we give a very brief introduction to the function spaces we use and techniques we employ in this thesis. For a more detailed approach we refer the reader to the books of Triebel \cite{triebel} and Pr\"uss and Simonett \cite{pruessbuch}.

\subsection{Bessel-Potential, Besov and Triebel-Lizorkin Spaces}

As usual, we will denote the classical $L_p$-Sobolev spaces on $\R^n$ by $W^k_p(\R^n)$, where $k$ is a natural number and $1 \leq p \leq \infty$. The Bessel-potential spaces will be denoted by $H^s_p(\R^n)$ for $s \in \R$ and the Sobolev-Slobodeckij spaces by $W^s_p(\R^n)$. We will also denote the usual Besov spaces by $B^s_{pr}(\R^n)$, where $s \in \R, 1 \leq p,r \leq \infty$. Lastly, as usual the Triebel-Lizorkin spaces are denoted by $F^s_{pr}(\R^n)$.

These function spaces on a domain $\Omega \subset \R^n$ are defined in a usual way by restriction. The Banach space-valued versions of these spaces are denoted by $L_{p}(\Omega;X)$, $W^{k}_{p}(\Omega;X)$, $H^{s}_{p}(\Omega;X)$, $W^{s}_{p}(\Omega;X)$, $B^{s}_{pr}(\Omega;X)$, $F^{s}_{pr}(\Omega;X)$, respectively. For precise definitions we refer to \cite{meyriesveraarpointwise}.

For results on embeddings, traces, interpolation and extension operators we refer to \cite{abelsbuch}, \cite{pruessbuch}, \cite{runst}, \cite{triebel}.

The following lemma is very well known and can easily be shown by using paraproduct estimates, see \cite{danchinbuch}.
\begin{lemma} \label{danchinlemma}
For any $s>0, 1 < p_1,r < \infty$,
\begin{equation} \label{paraestimate}
| vw |_{B^s_{p_1r}(\Rn)} \lesssim | v |_{B^s_{p_1r}(\Rn)} |w|_{L_\infty(\Rn)} + |v|_{L_\infty(\Rn)}| w |_{B^s_{p_1r} (\Rn)}
\end{equation}
for all $v,w \in B^s_{p_1r} (\Rn) \cap L_\infty(\Rn)$. In particular, the space $B^s_{p_1r} (\Rn) \cap L_\infty(\Rn)$ is an algebra.
\end{lemma}
\begin{proof}
See Corollary 2.86 in \cite{danchinbuch}.
\end{proof}

\subsection{$\mathcal R$-Boundedness, $\mathcal R$-Sectoriality and $\mathcal H^\infty$-Calculus}
We first define the notion of sectorial operators as in Definition 3.1.1 in \cite{pruessbuch}.
\begin{definition}
Let $X$ be a complex Banach space and $A$ be a closed linear operator on $X$. Then $A$ is said to be \textit{sectorial}, if both domain and range of $A$ are dense in $X$, the resolvent set of $A$ contains $(-\infty,0)$, and there is some $C>0$ such that $|t(t+A)^{-1}|_{\mathcal L(X)} \leq C$ for all $t>0.$
\end{definition}

The concept of $\mathcal R$-bounded families of operators is next. We refer to Definition 4.1.1 in \cite{pruessbuch}.
\begin{definition}
Let $X$ and $Y$ be Banach spaces and $\mathcal T \subseteq \mathcal L(X,Y)$. We say that $T$ is $\mathcal R$-bounded, if there is some $C>0$ and $p \in [1,\infty)$, such that for each $N \in \mathbb N, \{ T_j : j = 1,...,N \} \subseteq \mathcal T$, $\{ x_j : j = 1,...,N \} \subseteq X$ and for all independent, symmetric, $\pm 1$-valued random variables $\varepsilon_j$ on a probability space $(\Omega, \mathcal A, \mu)$ the inequality
\begin{equation} \label{rbound}
\left| \sum_{j=1}^N \varepsilon_j T_j x_j \right|_{L^p(\Omega;Y)} \leq C \left| \sum_{j=1}^N \varepsilon_j  x_j \right|_{L^p(\Omega;X)}
\end{equation}
is valid. The smallest $C>0$ such that \eqref{rbound} holds is called $\mathcal R$-bound of $\mathcal T$ and denote it by $\mathcal R(\mathcal T).$
\end{definition}

We can now define $\mathcal R$-sectoriality of an operator as is done in Definition 4.4.1 in \cite{pruessbuch}.
\begin{definition}
Let $X$ be a Banach space and $A$ a sectorial operator on $X$. It is then said to be $\mathcal R$-sectorial, if $\mathcal R_A(0) := \mathcal R \{ t(t+A)^{-1} : t > 0 \}$ is finite. We can then define the $\mathcal R$-angle of $A$ by means of $\phi_A^R := \inf \{ \theta \in (0,\pi) : \mathcal R_A (\pi - \theta) < \infty \}.$ Here, $\mathcal R_A ( \theta)  := \mathcal R \{ \lambda (\lambda + A)^{-1} : |\arg \lambda| \leq \theta \}$.
\end{definition}

We now define the important class of operators which admit a bounded $\mathcal H ^\infty$-calculus as in Definition 3.3.12 in \cite{pruessbuch}. For the well known Dunford functional calculus and an extension of which we refer to Sections 3.1.4 and 3.3.2 in \cite{pruessbuch}.
Let $0 < \phi \leq \pi$ and $\Sigma_\phi := \{ z \in \mathbb C : | \arg z | < \phi \}$ be the open sector with opening angle $\phi$. Let $H(\Sigma_\phi)$ be the set of all holomorphic functions $f : \Sigma_\phi \pfeil \mathbb C$ and $H^\infty(\Sigma_\phi)$ the subset of all bounded functions of $H(\Sigma_\phi).$ The norm in $H^\infty(\Sigma_\phi)$ is given by
\begin{equation}
| f |_{H^\infty (\Sigma_\phi)} := \sup \{ | f(z) | : z \in \Sigma_\phi \}.
\end{equation}
Furthermore let
\begin{equation}
H_0 (\Sigma_\phi) := \bigcup_{\alpha, \beta < 0} H_{\alpha,\beta}(\Sigma_\phi) ,
\end{equation}
where
$H_{\alpha,\beta}(\Sigma_\phi) := \{ f \in H(\Sigma_\phi) : |f|_{\alpha,\beta}^\phi < \infty \}$, and $|f|^\phi_{\alpha,\beta} := \sup \{ |z^\alpha f(z) | : |z| \leq 1 \} + \sup \{ |z^{-\beta} f(z) | : |z| \geq 1 \}$.
\begin{definition}
Let $X$ be a Banach space and $A$ a sectorial operator on $X$. Then $A$ admits a bounded $\mathcal H ^ \infty$-calculus, if there are $\phi > \phi_A$ and a constant $K_\phi < \infty$, such that
\begin{equation} \label{4958677589}
|f(A)|_{\mathcal L(X)} \leq K_\phi | f|_{H ^\infty (\Sigma_\phi)}
\end{equation}
for all $f \in H_0(\Sigma_\phi)$. The class of operators admitting a bounded $\mathcal H^\infty$-calculus on $X$ will be denoted by $\mathcal H^\infty (X)$. The $\mathcal H^\infty$-angle of $A$ is defined by the infimum of all $\phi > \phi_A$, such that \eqref{4958677589} is valid, $\phi_A^\infty := \inf \{ \phi > \phi_A : \eqref{4958677589} \text{ holds} \}.$
\end{definition}

\subsection{Maximal Regularity}
Let us recall the property of an operator having maximal $L_p$-regularity as is done in Definition 3.5.1 in \cite{pruessbuch}.
\begin{definition}
Let $X$ be a Banach space, $ J = (0,T), 0 < T < \infty$ or $J = \mathbb R_+$ and $A$ a closed, densely defined operator on $X$ with domain $D(A) \subseteq X.$ Then the operator $A$ is said to have maximal $L_p$-regularity on $J$, if and only if for every $f \in L_p(J;X)$ there is a unique $u \in W^1_p(J;X) \cap L_p(J;D(A))$ solving 
\begin{equation}
\frac{d}{dt}u (t) + A u (t)= f(t), \quad t \in J, \qquad u|_{t=0} = 0,
\end{equation}
in an almost-everywhere sense in $L_p(J;X)$.
\end{definition}
There is a wide class of results on operators having maximal regularity, we refer to sections 3.5 and 4 in \cite{pruessbuch} for further discussion. For results on $\mathcal R$-boundedness and interpolation we refer to \cite{kaipsaal}.

\section{Reduction to a Fixed Reference Surface}
In this section we transform the problem \eqref{9283486}-\eqref{9283486Z} to a fixed reference configuration. To this end we construct a suitable Hanzawa transfrom, taking into account the possibly curved boundary of $\partial\Omega$, by locally introducing curvilinear coordinates.

 Let $\Sigma \subset \Omega$ be a smooth reference surface and $\partial\Omega$ be smooth at least in a neighbourhood of $\partial\Sigma$. Furthermore, let $\angle(\Sigma,\partial\Omega) = \pi/2$ on $\partial\Sigma$. From Proposition 3.1 in \cite{vogel} we get the existence of so called curvilinear coordinates at least in a small neighbourhood of $\Sigma$, that is, there is some possibly small $a>0$ depending on the curvature of $\Sigma$ and $\partial\Omega$, such that
\begin{equation}
X : \Sigma \times (-a,a) \pfeil \R^n, \quad (p,w) \mapsto X(p,w),
\end{equation}
is a smooth diffeomorphism onto its image and $X(.,.)$ is a curvilinear coordinate system. This means in particular that points on the boundary $\partial\Omega$ only get transported along the boundary, $X(p,w) \in \partial\Omega$ for all $p \in \partial\Sigma, w \in (-a,a)$. We need to make use of these coordinates since the boundary $\partial\Omega$ may be curved. Therefore a transport only in normal direction of $n_\Sigma$ is not sufficient here. For details we refer to \cite{vogel}. 

With the help of these coordinates we may parametrize the free interface as follows.
We assume that at time $t \geq 0$, the free interface is given as a graph over the reference surface $\Sigma$, that is, there is some $h : \Sigma \times [0,T] \pfeil (-a,a)$, such that
\begin{equation}
\Gamma(t) = \Gamma_h (t) := \{ X(p,h(p,t)) : p \in \Sigma \}, \quad t \in [0,T],
\end{equation}
for small $T > 0$, at least.
With the help of this coordinate system we may construct a Hanzawa-type transform as follows.

Let $\chi \in C_0^\infty(\R)$ be a fixed function satisfying $\chi(s)=1$ for $|s| \leq 1/3$, $\chi(s) = 0$ for $|s| \geq 2/3$ and $|\chi'(s)|\leq 4$ for all $s \in \R$ and $\Sigma_a := X(\Sigma \times (-a,a))$. Then for a given height function $h : \Sigma \pfeil (-a,a)$ describing an interface $\Gamma_h$ we define
\begin{equation}
\Theta_h (x) := \begin{cases}
x, & x \not\in \Sigma_a, \\ 
(X \circ F_h \circ X^{-1} )(x), & x \in \Sigma_a, \end{cases}
\end{equation}
where
\begin{equation}
F_h (p,w) := \left( p, w - \chi( (w-h(p))/a ) h(p) \right), \quad p \in \Sigma, w \in (-a,a).
\end{equation}
Recall that by properties of the curvilinear coordinate system, we have $\Sigma = \{ x \in \R^n : x = X(p,0), p \in \Sigma \}.$ Let 
\begin{equation}
\mathcal U := \{ h \in X_\gamma : |h|_{L_\infty(\Sigma)} < a/5 \}.
\end{equation}
 Then we have the following result.
\begin{theorem}
For fixed $h \in \mathcal U$, the transformation $\Theta_h : \Omega \pfeil \Omega$ is a $C^1$-diffeomorphism satisfying $\Theta_h ( \Gamma_h ) = \Sigma$.
\end{theorem}
\begin{proof}
The proof is straightforward. It is easy to check that for $x \in \Gamma_h$ we have that $\Theta_h (x) = X(p,0)$, where $p \in \Sigma$ is determined by the identity $x = X(p,h(p))$. Hence $\Theta_h (\Gamma_h) = \Sigma$. Furthermore it is easy to see that $DF_h$ and hence $D\Theta_h$ is invertible in every point which concludes the proof since $X_\gamma \into C^2(\Sigma)$.
\end{proof}
The following lemma gives a decomposition of the transformed curvature operator $K(h) := H_{\Gamma_h} \circ \Theta_h$ for $h \in \mathcal U$. The result and proof are an adpation of the work in Lemma 2.1 in \cite{abelswilke} and Lemma 3.1 in \cite{eschersimonett}.
\begin{lemma} \label{9286fg349840}
Let $n=2,3$, $q \in (3/2,2), p > 3/(2-3/q)$ and $\mathcal U \subset X_\gamma$ be as before. Then there are functions
\begin{equation}
P \in C^1(\mathcal U, \mathcal B ( W^{4-1/q}_q(\Sigma), W^{2-1/q}_q (\Sigma)), \quad Q \in C^1( \mathcal U,  W^{2-1/q}_q (\Sigma)),
\end{equation}
such that
\begin{equation}
K(h) = P(h)h + Q(h), \quad \text{for all } h \in \mathcal U \cap W^{4-1/q}_q(\Sigma).
\end{equation}
Moreover,
\begin{equation}
P(0) = - \Delta_\Sigma,
\end{equation}
where $\Delta_\Sigma$ denotes the Laplace-Beltrami operator with respect to the surface $\Sigma$.
\end{lemma}
\begin{remark}
Note that the orthogonality relations (3.2) in \cite{eschersimonett} do not hold if we take $X$ to be curvilinear coordinates, since in $X$ we not only have a variation in normal but also in tangential direction. Therefore we have to modify the proofs in \cite{abelswilke},\cite{eschersimonett}. 
\end{remark}
\begin{proof}
The curvilinear coordinates $X$ are of form
\begin{equation} \label{93467858937465}
X = X(s,r) = s + r n_\Sigma(s) + \tau(s,r)\vec T(s), \quad s \in \Sigma, r \in (-a,a),
\end{equation}
where the tangential correction $\tau \vec T$ is as in \cite{vogel}. More precisely, $n_\Sigma$ denotes the unit normal vector field of $\Sigma$ with fixed orientation, $\vec T$ is a smooth vector field defined on the closure of $\Sigma$ with the following properties: it is tangent to $\Sigma$, normal to $\partial\Sigma$, of unit length on $\partial\Sigma$ and vanishing outside a neighbourhood of $\partial\Sigma$. In particular, $\vec T$ is bounded. Furthermore, $\tau = \tau(s,r)$ is a smooth scalar function such that $X(s,r)$ lies on $\partial\Omega$ whenever $s \in \partial\Sigma$. It satisfies $\tau(s,0)=0$ for all $s \in \Sigma$. Moreover, since $\Sigma$ and $\partial\Omega$ have a ninety degree contact angle, we have that
\begin{equation} \label{9237648634}
\partial_r \tau(s,0) = 0, \quad s \in \partial\Sigma.
\end{equation}
Hence we may choose $\tau$ in \cite{vogel} to satisfy \eqref{9237648634} for all $s \in \Sigma$. We will now derive a formula for the transformed mean curvature $K(h)$ in local coordinates. We follow the arguments of \cite{eschersimonett}.

The surface $\Gamma_h(t)$ is the zero level set of the function
\begin{equation}
\phi_h(x,t) := (X^{-1})_2(x) - h( (X^{-1})_1 (x),t), \quad x \in \Sigma_a, t \in \R_+,
\end{equation}
whence we define 
\begin{equation}
\Phi_h(s,r) := \phi_h ( X(s,r), t) = r - h(s,t), \quad s \in \Sigma, r \in (-a,a).
\end{equation}
We obtain that since $X : \Sigma \times (-a,a) \pfeil \R^n$ is a smooth diffeomorphism onto its image it induces a Riemannian metric $g_X$ on $\Sigma \times (-a,a)$. We denote the induced differential operators gradient, Laplace-Beltrami and the hessian with respect to $(\Sigma \times (-a,a), g_X)$ by $\nabla_X, \Delta_X$ and $\operatorname{hess}_X$. As in equation (3.1) in \cite{eschersimonett} we find that
\begin{equation}
K(h)|_s = \frac{1}{\| \nabla_X \Phi_h \|_X} \left( \Delta_X \Phi_h - \frac{[\operatorname{hess}_X \Phi_h] (\nabla_X \Phi_h, \nabla_X \Phi_h )}{\| \nabla_X \Phi_h \|_X^2 } \right)\big|_{(s,h(s))},
\end{equation}
for all $s \in \Sigma$, where $\| \nabla_X \Phi_h \|_X := ( g_X (\nabla_X \Phi_h, \nabla_X \Phi_h ))^{1/2}.$ Note at this point that since $X$ induces also a variation in tangential direction, the orthogonality relations (3.2) in \cite{eschersimonett} do not hold in general. However, we get in local coordinates that
\begin{equation}
( \partial_j X | \partial_n X ) = ( \partial_j X | n_\Sigma ) + \partial_r \tau ( \partial_j X | \vec T ), \quad j \in \{1,...,n-1 \},
\end{equation}
and $( \partial_n X | \partial_n X) = 1 + (\partial_r \tau)^2 ( \vec T | \vec T).$ In particular we see that on the surface $\Sigma$ the relations (3.2) in \cite{eschersimonett} still hold, but not away from $\Sigma$ in general. By using well-known representation formulas for $\nabla_X$, $\Delta_X$, and $\operatorname{hess}_X$ in local coordinates, one finds that
\begin{equation}
K(h)|_s = \left( \sum_{j,k = 1}^{n-1} a_{jk}(h) \partial_j \partial_k h + \sum_{j=1}^{n-1} a_j(h)\partial_j h + a(h) \right)|_{(s,h(s))},
\end{equation}
where
\begin{align}
a_{jk}(h) &= \frac{1}{\ell_X(h)^3} \bigg( - \ell_X (h)^2 w^{jk} + w^{jn} w^{kn} - \sum_{l=1} g^{jl}g^{kn} \partial_l h  \\ &- \sum_{l=1}^{n-1} g^{jn} g^{kl} \partial_l h + \sum_{l,m = 1}^{n-1} g^{jm} g^{kl} \partial_l h \partial_m h \bigg),
\end{align}
as well as
\begin{align}
a_j(h) &= \frac{1}{\ell_X(h)^3} \bigg( \ell_X(h)^2 \sum_{l,k =1}^n \Gamma^j_{lk} w^{lk}  - \sum_{q,k=1}^{n-1} \sum_{i,l=1}^n \Gamma^k_{il} w^{iq}w^{lj} \partial_k h \partial_q h \\
&+ \sum_{q=1}^{n-1} \sum_{i,l=1}^n \Gamma^n_{il} w^{iq} w^{lj} \partial_q h + \sum_{k=1}^{n-1} \sum_{i,l=1}^n \Gamma^k_{il} w^{in} w^{lj} \partial_k h - \sum_{i,l=1}^n \Gamma^n_{il} w^{in} w^{lj} \\
& + \sum_{k=1}^{n-1} \sum_{i,l=1}^n \Gamma^k_{il} w^{ij} w^{ln} \partial_k h 
- \sum_{i,l=1}^n \Gamma^n_{il} w^{ij} w^{ln} - \sum_{i,l=1}^n \Gamma^j_{il} w^{in} w^{ln}\bigg),
\end{align}
and 
\begin{equation}
a(h) =- \frac{1}{\ell_X(h)} \sum_{j,k=1}^n \Gamma^n_{jk} w^{jk} + \frac{1}{\ell_X(h)^3} \sum_{i,j=1}^n \Gamma^n_{ij} w^{in} w^{jn},
\end{equation}
where $w_{ij} := (\partial_i X| \partial_j X)$, $(w^{ij}) = (w_{ij})^{-1}$ and $\ell_X(h) := \| \nabla_X \Phi_h \|_X$.
Let
\begin{align} 
P(h)|_s &=  \left( \sum_{j,k = 1}^{n-1} a_{jk}(h) \partial_j \partial_k  + \sum_{j=1}^{n-1} a_j(h)\partial_j \right)|_{(s,h(s))}, \label{formulap} \\
Q(h)|_s &=    a(h)|_{(s,h(s))},
\end{align}
in local coordinates. Mimicking the proof of Lemma 2.1 in \cite{abelswilke}, $K(h) = P(h)h + Q(h)$ is the desired decompostion of $K$, since $X_\gamma \into C^2(\Sigma)$. The fact that $P(0) =- \Delta_\Sigma$ follows from \eqref{formulap} and the formulas for $a_{jk}$ and $a_j$.
\end{proof}
We are now able now transform the problem \eqref{9283486}-\eqref{9283486Z} to a fixed reference domain $\Omega \backslash \Sigma$ by means of the Hanzawa transform. This however yields a highly nonlinear problem for the height function. The transformed differential operators are given by
\begin{equation}
\nabla_h := (D\Theta_h^t)^\top \nabla, \quad \div_h u := \operatorname{Tr}(\nabla_h u), \quad \Delta_h := \div_h \nabla_h,
\end{equation}
and the transformed normal by $n_{\partial\Omega}^h := n_{\partial\Omega} \circ \Theta_h^t$. This leads to the equivalent system
\begin{subequations} \label{fullnonlienar8765}
\begin{align}
\partial_t h &= - \ljump n_{\Gamma_h(t)} \cdot \nabla_h \eta \rjump + (\beta(h)|n_{\Gamma_h(t)}-n_\Sigma), &&\text{on } \Sigma, \label{456356} \\
\eta|_{\Sigma} &= K(h), && \text{on } \Sigma, \\
\Delta_h \eta &= 0, && \text{in } \Omega \backslash \Sigma, \\
n_{\partial\Omega}^h \cdot \nabla_h \eta|_{\partial\Omega} &= 0, && \text{on } \partial\Omega, \\
n_{\partial\Omega}^h \cdot n_{\Gamma_h(t)} &= 0, && \text{on } \partial\Sigma, \label{2938dfg623} \\
h|_{t=0} &= h_0, && \text{on } \Sigma, \label{456356Z}
\end{align}
where $h_0$ is a suitable description of the initial configuration such that $\Gamma|_{t=0} = \Gamma_0$ and $\beta(h) := \partial_t h n_\Sigma + \partial_r \tau \vec T$, cf. \eqref{93467858937465}. Note that by the initial condition \eqref{9283486Z} we have that $n_{\partial\Omega}^{h_0} \cdot n_{\Gamma_{h_0}} = 0$, which is a necessary compatibility condition for the system \eqref{456356}-\eqref{456356Z}.
\end{subequations}

The following lemma states important differentiability properties of the transformed differential operators. 
\begin{lemma} \label{93845703746534}
Let $n=2,3$, $q \in (3/2,2)$, $p > 3/(3-4/q)$ and $\mathcal U \subset X_\gamma$ as before. Then
\begin{gather}
[ h \mapsto \Delta_h ] \in C^1( \mathcal U ; \mathcal B ( W^2_q (\Omega \backslash \Sigma ) ; L_q(\Omega))), \\
[ h \mapsto \nabla_h ] \in C^1( \mathcal U ; \mathcal B ( W^k_q (\Omega \backslash \Sigma ) ; W^{k-1}_q(\Omega \backslash \Sigma))), \quad k=1,2, \\
[ h \mapsto n_\Sigma^h ], \; [h \mapsto n_{\partial\Omega}^h] \in C^1(\mathcal U ;C^1(\Sigma)).
\end{gather}
\end{lemma}
\begin{proof}
The proof follows the lines of Section 4 in \cite{abelswilke}, since $X_\gamma \into C^2(\Sigma)$ by choice of $p$ and $q$.
\end{proof}
\section{Maximal $L_p-L_q$ Regularity for Linearized Problem}
\subsection{Reflection Operators}
We denote the upper half space of $\R^n$ by $\R^n_+ := \{ x \in \R^n : x_n > 0 \}$.
We will denote by $R$ the even reflection of a function defined on $\R^n_+$ across the boundary $\partial\R^n_+$ in $x_n$ direction, that is, we define $R$ as an extension operator via $Ru (t,x_1,...,x_n) := u(t,x_1,...,-x_n)$ for all $x_n < 0$. Note that $R$ admits a bounded extension $R: L_q(\R^n_+) \pfeil L_q(\R^n)$. The following theorems state that even more is true.
\begin{theorem} Let $1 < q < \infty$. \label{456745645}
The even reflection in $x_n$ direction $R$ induces a bounded linear operator from $W^{1+\alpha}_q(\R^n_+)$ to $W^{1+\alpha}_q(\R^n)$, whenever $0 \leq \alpha < 1/q$.
\end{theorem}
\begin{proof}
It is straightforward to verify that for a given $u \in W^{1+\alpha}_q(\R^n_+)$,
\begin{equation}
\partial_j Ru (x_1,...,x_n) = \partial_j u(x_1,...,-x_n), \quad j = 1,...,n-1, \; x_n < 0,
\end{equation}
and $\partial_n Ru(x_1,...,x_n) = - \partial_n u(x_1,...,-x_n)$. Hence also
 $R: W^{1}_q(\mathbb{R}^n_+)\rightarrow W^{1}_q(\mathbb{R}^n)$ is a bounded operator. To show the claim for the fractional order space of order $1+\alpha$, it remains to show that the odd reflection of $Du \in W^{\alpha}_q(\mathbb{R}^n_+)$, that is, say $TDu,$ is again $W^{\alpha}_q(\mathbb{R}^n)$ and that the corresponding bounds hold true.

We first note that $TDu(x_1,...,x_n) = e_0 Du(x_1,...,x_n) - e_0 Du(x_1,...,-x_n)$, where $e_0$ denotes the extension by zero to the lower half plane. 
Note that by real interpolation method, \[ W^\alpha_q (\rnplus) = \left( L_q( \rnplus ), W^1_{q,0} (\rnplus) \right)_{\alpha, q}, \quad 
W^\alpha_q (\rn) = \left( L_q( \rn ), W^1_{q} (\rn) \right)_{\alpha, q},\] 
since $0 < \alpha < 1/q,$ cf. \cite{triebel}. Now, both zero extension operators
\begin{equation}
e_0 : L_q(\rnplus) \rightarrow L_q(\rn) , \quad e_0 : W^1_{q,0}(\rnplus) \rightarrow W^1_{q}(\rn),
\end{equation}
are bounded and linear. From Theorem 1.1.6 in \cite{lunardiinterpol} we obtain that $e_0$ is therefore also a bounded and linear operator between the corresponding interpolation spaces, hence the theorem is proven.
\end{proof}
Note that the above proof makes essential use of the fact that the derivative of $u \in W^{1+\alpha}_q(\R^n_+)$ has no trace on $\partial \R^n_+$ since $\alpha < 1/q$. If one has a trace it needs to be zero to reflect appropriately, which is the statement of the next theorem. The proof follows similar lines, we omit it here.
\begin{theorem}
Let $q$ and $R$ be as above. Then $R$ induces a bounded linear operator 
\begin{equation}
W^{1+\beta}_q(\R^n_+) \cap \{ u : \partial_{x_n} u|_{x_n = 0} = 0 \} \pfeil W^{1+\beta}_q(\Rn)
\end{equation}
for all $\beta \in (1/q,1)$.
\end{theorem}
We also need a reflection argument for the initial data in $X_\gamma$. The result reads as follows.
\begin{theorem} 
The even reflection $R$ induces a bounded linear operator
\begin{equation}
W^{3+\alpha}_q(\R^n_+) \cap \{ u : \partial_{x_n} u|_{x_n = 0} = 0 \} \pfeil W^{3+\alpha}_q(\Rn)
\end{equation}
for all $\alpha \in (0,1/q),$ $q \in (3/2,2)$. In particular, $R$ also induces a bounded linear operator
\begin{equation}
B^{4-1/q-3/p}_{qp}(\R^n_+) \cap \{ u : \partial_{x_n} u|_{x_n = 0} = 0 \} \pfeil B^{4-1/q-3/p}_{qp}(\Rn)
\end{equation}
for all $q \in (3/2,2)$ and $p > 3/(2-3/q)$.
\end{theorem}
\begin{proof}
The second statement follows from the first one for $\alpha = 1 - 1/q - 3/p < 1/q$ since $q < 2$. The first claim is shown as in the proof of Theorem \ref{456745645}, using additionally that $\partial_{x_n} \partial_{x_n} Ru = R \partial_{x_n} \partial_{x_n} u$.
\end{proof}
\subsection{The Shifted Model Problem on the Half Space} \label{section34345439}
Let $n=2,3$. In this section we will be concerned with the linearized problem on the whole upper half space $\R^n_+$ with a flat interface $\Sigma := \{ x \in \R^n_+ : x_1 = 0 \}$. More precisely, we will consider
\begin{subequations} \label{fjnjj435743}
\begin{align} \label{9236749376454}
\partial_t h + \omega^3 h + \ljump n_\Sigma \cdot \nabla \mu \rjump &= g_1, && \text{on } \Sigma, \\
\mu|_\Sigma + \Delta_{x'} h &= g_2, && \text{on } \Sigma, \label{uosgdfouszgfdozusdgf} \\
\omega^2 \mu - \Delta \mu &= g_3, && \text{on } \R^n_+ \backslash \Sigma,  \label{8304765347} \\
e_n \cdot \nabla \mu |_{\partial \R^n_+} &= g_4, && \text{on } \partial\R^n_+, \label{8304765347A}\\
e_n \cdot \nabla_{x'} h |_{\partial\Sigma} &= g_5, && \text{on } \partial\Sigma, \\
h|_{t=0} &= h_0, && \text{on } \Sigma. \label{9236749376454Z}
\end{align}
\end{subequations}
Here, $x' = (x_2,...,x_n)$ and $\omega > 0$ is a fixed shift parameter we need to introduce to get maximal regularity results on the unbounded time-space domain $\R_+ \times \R^n_+ $.

Let us discuss the optimal regularity classes for the data. We seek a solution $h$ of this evolution equation in the space
\begin{equation}
W^1_p(\R_+; W^{1-1/q}_q (\Sigma)) \cap L_p(\R_+; W^{4-1/q}_q(\Sigma)),
\end{equation}
where $p$ and $q$ are specified below. In particular, $\mu \in L_p(\R_+; W^2_q(\R^n_+ \backslash \Sigma))$. Let 
\begin{equation}
X_0 := W^{1-1/q}_q (\Sigma), \quad X_1 := W^{4-1/q}_q (\Sigma),
\end{equation}
and the real interpolation space
\begin{equation}
X_\gamma := (X_1,X_0)_{1-1/p,p} = B^{4-1/q-3/p}_{qp}(\Sigma).
\end{equation}
By simple trace theory, we may deduce the necessary conditions
\begin{gather} \label{8926739873}
g_1 \in  L_p(\R_+;X_0), \quad g_2 \in L_p(\R_+; W^{2-1/q}_q(\Sigma)), \\ g_3 \in L_p(\R_+;L_q(\R^n_+)), \quad g_4 \in L_p(\R_+; W^{1-1/q}_q(\partial\R^n_+)), \quad h_0 \in X_\gamma.
\end{gather}
It is now a delicate matter to find the optimal regularity condition for $g_5$, which turns out to be
\begin{equation}
g_5 \in F^{1-2/(3q)}_{pq} (\R_+; L_q (\partial\Sigma)) \cap L_p(\R_+ ; W^{3-2/q}_{q} (\partial\Sigma)),
\label{8926739873Z}
\end{equation}
cf. Theorem \ref{8374650783645} in the Appendix. Note that $g_5$ has a time trace at $t=0$, whenever $1-2/(3q) - 1/p > 0$. Hence there is a compatibility condition inside the system whenever this inequality is satisfied, namely
\begin{equation} \label{8926739873C}
g_5|_{t=0} = e_n \cdot \nabla_{x'} h_0 |_{\partial\Sigma} = \partial_n h_0 |_{\partial\Sigma}, \quad \text{on } \partial\Sigma.
\end{equation}
Note that there is no compatibility condition stemming from \eqref{uosgdfouszgfdozusdgf} and \eqref{8304765347A} on $\partial\Sigma$, whenever $q <2$.
The following theorem now states that these conditions are also sufficient. Note that the assumptions in Theorem \ref{93453845437} imply that $q<2$ and $1-2/(3q) - 1/p > 0$ hold.
\begin{theorem} \label{93453845437}
Let $6 \leq p < \infty$, $q \in (3/2,2) \cap (2p/(p+1),2p)$ and $\omega > 0$. Then \eqref{9236749376454}-\eqref{9236749376454Z} has maximal $L_p-L_q$-regularity on $\R_+$, that is, for every $(g_1,g_2,g_3,g_4,g_5,h_0)$ satisfying the regularity conditions \eqref{8926739873}-\eqref{8926739873Z} and the compatibility condition \eqref{8926739873C}, there is a unique solution $(h,\mu) \in (W^1_p(\R_+;X_0) \cap L_p(\R_+;X_1) ) \times L_p(\R_+;W^2_q(\R^n_+ \backslash \Sigma))$ of the shifted half space problem \eqref{9236749376454}-\eqref{9236749376454Z}.

Furthermore,
\begin{equation}
|h|_{W^1_p(\R_+;X_0) \cap L_p(\R_+;X_1)} + |\mu|_{L_p(\R_+;W^2_q(\R^n_+ \backslash \Sigma))}
\end{equation}
is bounded by
\begin{gather*}
|g_1|_{ L_p(\R_+;X_0)} + |g_2|_{L_p(\R_+; W^{2-1/q}_q(\Sigma))} + |g_3|_{L_p(\R_+;L_q(\R^n_+))} + \\
|g_4|_{L_p(\R_+; W^{1-1/q}_q(\partial\R^n_+))} + |g_5|_{F^{1-2/(3q)}_{pq} (\R_+; L_q (\partial\Sigma)) \cap L_p(\R_+ ; W^{3-2/q}_{q} (\partial\Sigma))} + |h_0|_{X_\gamma}
\end{gather*}
up to a constant $C=C(\omega) > 0$ which may depend on $\omega >0$.
\end{theorem}
\begin{proof}
We first reduce to a trivial initial value by extending $h_0$ to $\tilde \Sigma = \{0\} \times \R^{n-1}$ using standard extension results of \cite{triebel} and solving an $L_p-L_q$ auxiliary problem on $\R^{n-1}$ using results of Section 4 in \cite{pruessmulsek} to find some $h_S \in W^1_p(\R_+;X_0) \cap L_p(\R_+;X_1)$ such that $h_S|_{t=0} = h_0$, cf. problem \eqref{90283640723650837465083476}. Then define $\tilde g_5 := g_5 - \partial_n h_S|_{\partial\Sigma}$. Clearly,
\begin{equation}
\tilde g_5|_{t=0} = g_5|_{t=0} - \partial_n h_0|_{\partial\Sigma} = 0, \quad \text{on } \partial\Sigma,
\end{equation}
by the compatibility condition \eqref{8926739873C}. This allows us to use Theorem \ref{8374650783645} to find some $\tilde h \in{_0 W}^1_p(\R_+;X_0) \cap L_p(\R_+;X_1)$ such that
\begin{equation}
\partial_n \tilde h|_{\partial\Sigma} = \tilde g_5, \quad \text{on } \partial\Sigma.
\end{equation}
  By simple trace theory we may find $\mu_4 \in L_p(\R_+; W^2_q(\R^n_+ \backslash \Sigma))$ such that $\partial_n \mu_4|_{\partial\R^n_+} = g_4$ on $\partial\R^n_+$. Let $\tilde \Sigma  := R\Sigma := \{ x \in \R^n : x_1 = 0 \}.$ We then solve the elliptic auxiliary problem
\begin{subequations} \label{92fddfg36749873}
\begin{align}
\omega^2 \tilde \mu - \Delta \tilde \mu &= Rg_3 - R\Delta \mu_4, && \text{on } \R^n \backslash \tilde \Sigma, \\
\tilde \mu|_{ \tilde \Sigma} &= R\Delta_{x'} \tilde h + R\Delta_{x'} h_S + Rg_2 - R\mu_4|_{\tilde\Sigma}, && \text{on } \tilde\Sigma, \label{973864587645}
\end{align}
\end{subequations}
by a unique $\tilde \mu \in L_p(\R_+;W^2_q(\R^n \backslash \tilde\Sigma))$, cf. \cite{amann}. Note at this point that we used that due to $q <2$ and Theorem \ref{456745645} we have that the data in \eqref{973864587645} is in $L_p(\R_+;W^{2-1/q}_q(\tilde\Sigma))$. Note that by construction $\tilde \mu$ is even in $x_n$ direction since both the data in \eqref{92fddfg36749873} are.

We have reduced the problem to the case where $(g_2,g_3,g_4,g_5,h_0) = 0$, that is, we are left to solve
\begin{subequations}
\begin{align} \label{923674937dfgdfg6454}
\partial_t h + \omega^3 h + \ljump n_\Sigma \cdot \nabla \mu \rjump &= g_1, && \text{on } \Sigma, \\
\mu|_\Sigma + \Delta_{x'} h &= 0, && \text{on } \Sigma, \\
\omega^2 \mu - \Delta \mu &= 0, && \text{on } \R^n_+ \backslash \Sigma,  \label{8304ergrg765347} \\
e_n \cdot \nabla \mu |_{\partial \R^n_+} &= 0, && \text{on } \partial\R^n_+, \label{8304dfgd765347A}\\
e_n \cdot \nabla_{x'} h |_{\partial\Sigma} &= 0, && \text{on } \partial\Sigma, \label{8304dfgd765347A2} \\
h|_{t=0} &= 0, && \text{on } \Sigma, \label{92367dfg49376454Z}
\end{align}
\end{subequations}
for possibly modified $g_1$ not to be relabeled in an $L_p-L_q$-setting. We reflect the problem once more across the boundary $\partial\R^n_+$ using the even reflection in $x_n$ direction $R$ and by doing so we obtain a full space problem with a flat interface and that the conditions \eqref{8304dfgd765347A} and \eqref{8304dfgd765347A2} are fulfilled automatically. We obtain the problem
\begin{subequations}
\begin{align} \label{grthrth54}
\partial_t h + \omega^3 h + \ljump n_\Sigma \cdot \nabla \mu \rjump &= Rg_1, && \text{on }\tilde \Sigma, \\
\mu|_{\tilde \Sigma} + \Delta_{x'} h &= 0, && \text{on } \tilde\Sigma,  \label{923674937dfgdddfg6454B}\\
\omega^2 \mu - \Delta \mu &= 0, && \text{on } \R^n \backslash \tilde \Sigma,   \label{923674937dfgdddfg6454} \\
h|_{t=0} &= 0, && \text{on } \tilde\Sigma, \label{92367dfg493ff76454Z}
\end{align}
\end{subequations}
where $Rg_1 \in L_p(\R_+;W^{1-1/q}_q(\tilde\Sigma))$. Let us denote by $S(h)$ the unique solution of the elliptic problem \eqref{923674937dfgdddfg6454B}-\eqref{923674937dfgdddfg6454}. Then we can write the system as an abstract evolution equation as follows. Define $\mathcal A h (x) := \ljump n_\Sigma \cdot \nabla S(h) \rjump - \omega^3 h$ and its realization in $W^{1-1/q}_q(\tilde\Sigma)$ by $ A : D(A) \pfeil W^{1-1/q}_q(\tilde\Sigma)$, where the domain of $A$ is given by
\begin{equation}
D(A) := W^{4-1/q}_q(\tilde\Sigma).
\end{equation}
Then we can modify the results of \cite{pruessmulsek} to obtain that the operator $A$ has the property of maximal $L_q$-regularity on the whole half line $\R_+$, whence a general principle of maximal regularity going back to Dore \cite{dore2000} and the works of Bourgain and Benedek, Calderon and Panzone \cite{bourgain1984} now gives that $A$ has also maximal $L_p$-regularity on $\R_+$, since $1 < p < \infty$, cf. \cite{pruessbuch}. We give the full details below. Having this at hand we can solve the initial value problem
\begin{subequations} \label{90283640723650837465083476}
\begin{align}
\frac{d}{dt}h(t) + Ah(t) &=\tilde f(t), \quad t \in \R_+, \\
h(0) &= \tilde h_0,
\end{align}
\end{subequations}
for any $\tilde f \in L_p(\R_+;W^{1-1/q}_q(\tilde\Sigma))$ and $\tilde h_0 \in B^{4-1/q-3/p}_{qp}(\tilde\Sigma)$ by a unique function $h \in W^1_p(\R_+;W^{1-1/q}_q(\tilde\Sigma)) \cap L_p(\R_+; W^{4-1/q}_q(\tilde\Sigma))$. By choosing
\begin{equation}
f := R\ljump n_\Sigma \cdot \nabla (\tilde \mu + \mu_4) \rjump - R\partial_t (\tilde h + h_S) + Rg_1, \quad \tilde h_0 := 0,
\end{equation}
we obtain a unique solution $(h,S(h))$ of the problem \eqref{grthrth54}-\eqref{92367dfg493ff76454Z} in the proper $L_p-L_q$-regularity classes on $\R_+ \times \R^{n-1}$. The estimate easily follows and the proof is complete.

Let us give the details on how we 
obtain maximal $L_q$-regularity for $A$ on $\R_+$. We take Fourier transform with respect to $(x_2,...,x_n) \in \R^{n-1}$ to obtain a system
\begin{subequations}
	\begin{align}
	\omega^3 \hat h + \partial_t \hat h + \ljump \partial_1 \hat \pi \rjump &=\hat f, && \xi \in \R^{n-1}, \\
	\omega^2 \hat \pi + |\xi|^2 \hat \pi - \partial_1^2 \hat \pi &= 0, && (x_1,\xi) \in \dot \R \times \R^{n-1}, \\
	\hat \pi|_{x_1=0} + |\xi|^2 \hat h &=0, && \xi \in \R^{n-1}, \\
	\hat h|_{t=0} &= 0, && \xi \in \R^{n-1},
	\end{align}
\end{subequations}
where $\hat \pi = \hat \pi(t,x_1,\xi)$, $\hat h = \hat h(t,\xi)$ and $\hat f = \hat f(t,\xi)$ denote the Fourier transforms of $\pi$, $h$, and $f$ with respect to the last $n-1$ variables $(x_2,...,x_n) \in \R^{n-1}$. We can now solve the second order differential equation for $\hat \pi$ and together with boundary and decay conditions we finally obtain
\begin{equation}
\ljump \partial_1 \hat \pi \rjump = 2|\xi|^2  \sqrt{\omega^2 + |\xi|^2}  \hat h,
\end{equation}
whence we obtain a modified version of the evolution equation in \cite{pruessmulsek}, namely
\begin{align}
(\partial_t + \omega^3) \hat h + \left( 2 |\xi|^2 \sqrt{\omega^2 + |\xi|^2} \right) \hat h &= \hat f, \quad t \in \R_+, \\   \hat h(t=0) &= 0.
\end{align}
Let now $B_1$ be the negative Laplacian on $L_q(\R^{n-1})$ with domain $W^2_q(\R^{n-1})$. It is now well known that $B_1$ admits an $\mathcal R$-bounded $\mathcal H^\infty$-calculus on $L_q(\R^{n-1})$ with corresponding $\mathcal {RH}^\infty$ angle zero, $\phi_{B_1}^{\mathcal {RH}^\infty} = 0$, cf. the proof of Proposition 8.3.1 in \cite{pruessbuch}. Let furthermore $B_2$ be the operator given by $(\omega^2 - \Delta)^{1/2}$ on $L_q(\R^{n-1})$ with natural domain $W^1_q(\R^{n-1})$. Then by Example 4.5.16(i) in \cite{pruessbuch} we know that $B_2$ is invertible and admits a bounded $\mathcal H^\infty$-calculus on $L_q(\R^{n-1})$ and the $\mathcal H^\infty$-angle is zero, $\phi_{B_2}^\infty = 0$. We now apply Corollary 4.5.12(iii) in \cite{pruessbuch} to get that $P := 2B_1B_2$ 
 is a closed, sectorial operator which itself admits a bounded $\mathcal H^\infty$-calculus on $L_q(\R^{n-1})$ as well and that the $\mathcal H^\infty$-angle of $P$ is zero. The fact that $B_1$ and $B_2$ commute stems from the fact that these are given as Fourier multiplication operators.

We now show that $P$ admits a bounded $\mathcal H^\infty$-calculus also on $W^s_q(\R^{n-1})$ for all $0 < s < 1$, in particular for $s=1-1/q$. To this end we show the claim for $s=1$ and use real interpolation method. We will use the fact that $(I-\Delta)^{ 1/2}$ is a bounded isomorphism from $W^1_q(\R^{n-1}) \pfeil L_q(\R^{n-1})$ with inverse $(I-\Delta)^{ -1/2}$.

Let $ \phi > 0$ and $\Sigma_\phi$ be the sector in the complex plane of opening angle $\phi$. Since $P$ admits a bounded $\mathcal H^\infty$-calculus on $L_q(\R^{n-1})$, there is a constant $K_\phi > 0$, such that
\begin{equation} \label{hinf5656}
| h (P)|_{\mathcal B ( L_q (\mathbb R^{n-1} ))} \leq K_\phi |h|_{H^\infty(\Sigma_\phi)}
\end{equation}
for all $h \in H_0(\Sigma_\phi).$ Let $u \in C^\infty_0(\mathbb R^{n-1}).$ Taking Fourier transform just as in the proof of Theorem 6.1.8 in \cite{pruessbuch} gives
\begin{equation}
\mathcal F[ h(P) u] (\xi) = h(\mathcal P(\xi)) \mathcal F u (\xi),
\end{equation}
where $\mathcal P(\xi) = 2|\xi|^2\sqrt{\omega^2 + |\xi|^2}$ is the corresponding symbol of $P$. Whence clearly we have the representation formula
\begin{equation}
h(P) u = \mathcal F^{-1} [h(\mathcal P(\xi)) \mathcal F u]
\end{equation}
for all $u \in C^\infty_0(\mathbb R^{n-1})$, in other words, the symbol of $h(P)$ is in fact $h(\mathcal P(\xi)).$ Since $h(P)$ and the shift operators $(I-\Delta)^{\pm 1/2}$ commute we easily see that $P$ admits a bounded $\mathcal H^\infty$-calculus on $W^1_q(\R^{n-1})$ and hence on all $W^s_q(\R^{n-1})$, where $0 < s <1 $. The constant extension to $L_p(\R_+;W^s_q(\R^{n-1}))$ which we will also denote by $P$ then admits a bounded $\mathcal H^\infty$-calculus on $L_p(\R_+;W^s_q(\R^{n-1}))$ for all $0< s<1$ with angle zero.

We now apply a version of Dore-Venni theorem, cf. \cite{REFPAPERPRUESS}. To this end let $B$ be the operator on $L_p(\R_+;W^{1-1/q}_q(\R^{n-1}))$ defined by $B = d/dt + \omega^3$ with domain
\begin{equation}
D(B) = {_0} W^1_p(\R_+;W^{1-1/q}_q(\R^{n-1})).
\end{equation}
Then $B$ is sectorial and admits a bounded $\mathcal H^\infty$-calculus on $L_p(\R_+;W^s_q(\R^{n-1}))$ of angle $\pi/2$. Furthermore, $B: D(B) \pfeil L_p(\R_+;W^s_q(\R^{n-1}))$ is invertible.
Let as above $P$ be the operator on $L_p(\R_+;W^{1-1/q}_q(\R^{n-1}))$ with domain $D(P) = L_p(\R_+;W^{4-1/q}_q(\R^{n-1}))$ given by its symbol $2|\xi|^2(\omega^2+|\xi|^2)^{1/2}$.
 Now, by the Dore-Venni theorem we get that the sum $B+P$ with domain $D(B+P)=D(B) \cap D(P)$ is closed, sectorial and invertible. In other words, the evolution equation $Bu+Pu =f$ posesses for every $f  \in L_p(\R_+;W^{1-1/q}_q(\R^{n-1}))$ a unique solution $u \in D(B) \cap D(P)$, hence the proof of maximal regularity is complete.
\end{proof}
\textbf{Dependence of the maximal regularity constant on the shift parameter.}
Note that at this point it is a priori not clear how the maximal regularity constant depends on the shift parameter $\omega > 0$. However, we will need a good understanding of this dependence later on when we want to solve the bent halfspace problems. 

We will now introduce suitable $\omega$-dependent norms in both data and solution space and show that the maximal regularity constant is then independent of $\omega$.

To this end we will proceed with a scaling argument. Fix $\omega> 0$ and let $(h,\mu)$ be the solution on $\R_+$ of the $\omega$-shifted half space problem \eqref{9236749376454}-\eqref{9236749376454Z}. Define new functions
\begin{equation}
  \tilde h(x,t) := \omega^2 h(x/\omega,t/\omega^3),\quad \tilde \mu (x,t) := \mu(x/\omega,t/\omega^3), \quad x \in \R_+, t \in \R_+.
\end{equation}
It is then easy to check that $(\tilde h,\tilde \mu)$ solves 
\begin{subequations}
	\begin{align}
	\partial_t \tilde h +  \tilde h + \ljump n_\Sigma \cdot \nabla \tilde\mu \rjump &= \tilde g_1, && \text{on } \Sigma, \\
	\tilde \mu|_\Sigma + \Delta_{x'} \tilde h &= \tilde g_2, && \text{on } \Sigma, \\
 \tilde \mu - \Delta \tilde \mu &= \tilde g_3, && \text{on } \R^n_+ \backslash \Sigma,   \\
	e_n \cdot \nabla \tilde \mu |_{\partial \R^n_+} &= \tilde g_4, && \text{on } \partial\R^n_+, \\
	e_n \cdot \nabla_{x'} \tilde h |_{\partial\Sigma} &= \tilde g_5, && \text{on } \partial\Sigma, \\
	\tilde h|_{t=0} &= \tilde h_0, && \text{on } \Sigma,
	\end{align}
\end{subequations}
where
\begin{gather}
\tilde g_1 (x,t) := \omega^{-1} g_1(x/\omega, t/\omega^3), \quad
\tilde g_2 (x,t) := g_2(x/\omega, t/\omega^3), \quad \\
\tilde g_3 (x,t) := \omega^{-2} g_3(x/\omega, t/\omega^3), 
\tilde g_4 (x,t) := \omega^{-1} g_4(x/\omega, t/\omega^3), \\
\tilde g_5 (x,t) := \omega g_5(x/\omega, t/\omega^3),
\tilde h_0 (x) := \omega^2 h_0 (x/\omega), \quad x \in \R_+, t \in \R_+.
\end{gather}
Since the operator on the left hand side is independent of $\omega$, we get by the previous theorem that there is some constant $M > 0$ independent of $\omega$, such that
\begin{equation} \label{9346858347654}
|\tilde h|_{W^1_p(\R_+;X_0) \cap L_p(\R_+;X_1)} + |\tilde \mu|_{L_p(\R_+;W^2_q(\R^n_+ \backslash \Sigma))} \end{equation} 
is bounded by
\begin{align}
& M \big(
|\tilde g_1|_{ L_p(\R_+;X_0)} + |\tilde g_2|_{L_p(\R_+; W^{2-1/q}_q(\Sigma))} + |\tilde g_3|_{L_p(\R_+;L_q(\R^n_+))} + \\
&+|\tilde g_4|_{L_p(\R_+; W^{1-1/q}_q(\partial\R^n_+))} + |\tilde g_5|_{F^{1-2/(3q)}_{pq} (\R_+; L_q (\partial\Sigma)) \cap L_p(\R_+ ; W^{3-2/q}_{q} (\partial\Sigma))} + |\tilde h_0|_{X_\gamma} \nonumber
\big).
\end{align}
Clearly, the $\omega$-dependence is now hidden in the norms, whenceforth a careful calculation entails
\begin{align*}
&\omega^{ 4-1/q    } | h |_{L_p(\R_+;L^q(\Sigma))} + \omega^{ 3-1/q    } | D h |_{L_p(\R_+;L_q(\Sigma))} + 
\omega^{ 2-1/q    } |D^2 h |_{L_p(\R_+;L_q(\Sigma))} + \\ &+ \omega^{ 1-1/q    } | D^3 h |_{L_p(\R_+;L_q(\Sigma))} + 
| [\partial_t h ]_{X_0} |_{L_p( \R_+)} + | [D^3 h ]_{X_0} |_{L_p( \R_+)}  + \\ &+ \omega^2 | \mu |_{L^p( \R_+;L_q(\rnplus \backslash \Sigma))} + \omega |D  \mu |_{L_p( \R_+;L_q(\rnplus \backslash \Sigma))} +  | D^2 \mu |_{L_p( \R_+;L_q(\rnplus \backslash \Sigma))} \leq \\
&\leq M \bigg( \omega^{1-1/q} | g_1 |_{L_p( \R_+;L_q(\Sigma))} +
| [g_1]_{X_0} |_{L_p( \R_+)} +
\omega^{2-1/q} | g_2 |_{L_p( \R_+;L^q(\Sigma))} + \\
&+  \omega^{1-1/q} | D g_2 |_{L_p( \R_+;L_q(\Sigma))} +   | [Dg_2]_{X_0} |_{L_p( \R_+)} + | g_3 |_{L_p( \R_+;L_q(\rnplus \backslash \Sigma))} + \\
&+ \omega^{1-1/q} | g_4 |_{L_p( \R_+;L_q (\partial \rnplus))} + | [g_4]_{W^{1-1/q}_q(\partial\rnplus)} |_{L_p( \R_+)} + \\
&+ \omega^{3-2/q} | g_5 |_{L_p( \R_+,L_q(\partial\Sigma))} 
+ \omega^{2-2/q} | D g_5 |_{L_p( \R_+,L_q(\partial\Sigma))}  \\
&+ | [Dg_5]_{W^{2-2/q}_q (\partial\Sigma)} |_{L_p( \R_+)} +
 [g_5]_{F^{1- \frac{2}{3q}}_{pq} ( \R_+;L_q(\partial\Sigma))}
+ K(\omega) | h_0 |_{X_{\gamma}}
\bigg),
\end{align*}
for some $K(\omega) >0$ stemming from interpolating the estimates for $X_0$ and $X_1$, the value of which does not matter. The calculations involving the Triebel-Lizorkin seminorm of $\tilde g_5$ stem from the characterization via differences, see Proposition 2.3 in \cite{meyriesveraartraces}.

We now define norms as follows. Let
\begin{align*}
|h|_{E,1,\omega} &:= \omega^{ 4-1/q    } | h |_{L_p(\R_+;L^q(\Sigma))} + \omega^{ 3-1/q    } | D h |_{L_p(\R_+;L_q(\Sigma))} + 
\omega^{ 2-1/q    } |D^2 h |_{L_p(\R_+;L_q(\Sigma))} + \\ &+ \omega^{ 1-1/q    } | D^3 h |_{L_p(\R_+;L_q(\Sigma))} + 
| [\partial_t h ]_{X_0} |_{L_p( \R_+)} + | [D^3 h ]_{X_0} |_{L_p( \R_+)} , \\
|\mu |_{E,2,\omega} &:= \omega^2 | \mu |_{L^p( \R_+;L_q(\rnplus \backslash \Sigma))} + \omega |D  \mu |_{L_p( \R_+;L_q(\rnplus \backslash \Sigma))} +  | D^2 \mu |_{L_p( \R_+;L_q(\rnplus \backslash \Sigma))}, \\
| g_1 |_{F,1,\omega} &:=\omega^{1-1/q} | g_1 |_{L_p( \R_+;L_q(\Sigma))} +
| [g_1]_{X_0} |_{L_p( \R_+)}, \\
| g_2 |_{F,2,\omega} &:= \omega^{2-1/q} | g_2 |_{L_p( \R_+;L^q(\Sigma))} + \omega^{1-1/q} | D g_2 |_{L_p( \R_+;L_q(\Sigma))} +   | [Dg_2]_{X_0} |_{L_p( \R_+)} , \\
| g_3 |_{F,3,\omega} &:= | g_3 |_{L_p( \R_+;L_q(\rnplus \backslash \Sigma))} , \\
| g_4|_{F,4,\omega} &:= \omega^{1-1/q} | g_4 |_{L_p( \R_+;L_q (\partial \rnplus))} + | [g_4]_{W^{1-1/q}_q(\partial\rnplus)} |_{L_p( \R_+)}, \\
| g_5|_{F,5,\omega} &:= \omega^{3-2/q} | g_5 |_{L_p( \R_+,L_q(\partial\Sigma))} 
+ \omega^{2-2/q} | D g_5 |_{L_p( \R_+,L_q(\partial\Sigma))}  \\
&+ | [Dg_5]_{W^{2-2/q}_q (\partial\Sigma)} |_{L_p( \R_+)} +
[g_5]_{F^{1- \frac{2}{3q}}_{pq} ( \R_+;L_q(\partial\Sigma))}, \end{align*} and 
$|h_0|_{F,6,\omega} := K(\omega) | h_0 |_{X_{\gamma}}$.
This way we obtain that $|h|_{E,1,\omega} + |\mu |_{E,2,\omega}$ is bounded by
\begin{equation}
 M ( |g_1|_{F,1,\omega} +|g_2|_{F,2,\omega}+|g_3|_{F,3,\omega}+|g_4|_{F,4,\omega}+ |g_5|_{F,5,\omega} + |h_0|_{F,6,\omega} ),
\end{equation}
where we point out that $M>0$ is independent of $\omega > 0$. Note that this estimate also holds true on bounded intervals $J = (0,T) \subset \R_+$, as can be seen as follows. First again reduce to trivial initial data as in the proof of Theorem \ref{93453845437}. Then one can simply extend the data $(g_1,g_2,g_3,g_4)$ to the half line $\R_+$ by zero. Regarding $g_5$ we note that after the reduction procedure, $g_5|_{t=0}=0$, whence we may use Section 3.4.3 in \cite{triebel} and Corollary 5.12 in \cite{kaipdiss} to extend $g_5$ to a function on the half line $\R_+$. Then on $J$ the same estimate holds true if we replace $M$ by $2M$.
\subsection{Bent Half Space Problems}
In this section we consider the shifted model problem \eqref{fjnjj435743} on a bent half space $\R^n_\gamma := \{ x \in \R^n : x_n > \gamma(x_1,...,x_{n-1}) \}$, where $\gamma : \R^{n-1} \pfeil \R$ is a sufficiently smooth function with sufficiently small $C^1(\R^{n-1})$-norm. Since also the reference surface may be curved, we consider a slightly bent interface $\Sigma_{\beta} := \{ x \in \overline{\R^n_\gamma} : x_1 = {\beta(x_2,...,x_n)} \}$. Again, $\beta : \R^{n-1} \pfeil \R$ is suitably smooth and the $C^1(\R^{n-1})$-norm is sufficiently small. The bent half space problem reads as
\begin{subequations} \label{fjnjj43t5743}
	\begin{align} \label{92367493764z54}
	\partial_t h + \omega^3 h + \ljump n_{\Sigma_{\beta}} \cdot \nabla \mu \rjump &= g_1, && \text{on } \Sigma_\beta, \\
	\mu|_{\Sigma_\beta} + \Delta_{\Sigma_\beta} h &= g_2, && \text{on } \Sigma_\beta, \\
	\omega^2 \mu - \Delta_x \mu &= g_3, & &\text{on } \R^n_\gamma \backslash \Sigma_\beta,  \label{8304765g347} \\
	n_\gamma \cdot \nabla \mu |_{\partial \R^n_\gamma} &= g_4, && \text{on } \partial\R^n_\gamma, \label{83047c65347A}\\
	n_\gamma \cdot \nabla_{\Sigma_\beta} h |_{\partial\Sigma_\beta} &= g_5, && \text{on } \partial\Sigma_\beta, \\
	h|_{t=0} &= h_0, && \text{on } \Sigma_\beta, \label{923674937d6454Z}
	\end{align}
\end{subequations}
where $n_\gamma$ denotes the normal at $\partial\R^n_\gamma$. The smallness assumption on $|\beta|_{C^1} + |\gamma|_{C^1}$ implies that the bent domain and interface are only a small perturbation of the half space and the flat interface. We will now solve this problem on the bent half space by transforming it back to the regular half space. 
\begin{lemma}
	Let $k \in \mathbb N$ and $\beta,\gamma \in C^k(\R^{n-1})$. Then there is some $F \in C^k(\R^n;\R^n ),$ such that $F: \R^n \pfeil \R^n$ is a $C^k$-diffeomorphism and such that additionally, $F|_{\R^n _\gamma } : \R^n _\gamma \pfeil \R^n _+$ is a $C^k$-diffeomorphism as well. Furthermore, $F$ maps $\Sigma_\beta$ to the flat interface $\R^n _+ \cap \{ x_1 = 0\}$. We also have that $| I - DF|_{C^l(\R^n)} \lesssim | \beta|_{C^{l+1}(\R^{n-1})} + | \gamma|_{C^{l+1}(\R^{n-1})}$ for all $l=0,...,k-1$.
\end{lemma}
\begin{proof} To economize notation, let $n = 3$.
	We first transform in $x_3$-direction via $\Phi_1 : \R^3 \pfeil \R^3, x \mapsto (x_1,x_2,x_3-\gamma(x_1,x_2))$. It is then easy to see that the surface $\Phi_1(\Sigma_\beta)$ 
	is given by the set
	\begin{equation}
	\Phi_1 (\Sigma_\beta) = \{ ( \beta(x_2,x_3) , x_2 ,x_3 - \gamma ( \beta (x_2,x_3) , x_2) ) : x_2 \in \R \} \cap \R^3_+.
	\end{equation}
	Note that this is equivalent to
	\begin{equation}
	\Phi_1 (\Sigma_\beta) =\{ ( \beta (x_2,x_3) , H(x_2,x_3) ) : (x_2,x_3) \in \R^2 \} \cap \R^3_+,
	\end{equation}
	where
	\begin{equation}
	H : \R^2 \pfeil \R^2, (x_2,x_3) \mapsto (x_2, x_3 - \gamma( \beta(x_2,x_3) , x_2 ) ).
	\end{equation}
	Now note that whenever $| ( \beta, \gamma) |_{C^1}$ is sufficiently small, $|H-\text{id}_{\mathbb R^2}|_{C^1}$ is small. Then $|\det DH |  \geq 1/2$ on $\R^2$ and $H : \R^2 \pfeil \R^2$ is globally invertible. Hence the surface $\Phi_1 (\Sigma_\beta)$ can be parametrized by $\beta \circ H^{-1}$,
	\begin{equation}
	\Phi_1 (\Sigma_\beta) = \{ ( \beta ( H^{-1} (x_2,x_3) ) , x_2 ,x_3 ) : (x_2,x_3) \in \R^2 \} \cap \R^3_+.
	\end{equation} Note that by the inverse function theorem, $H^{-1}$ is $C^1(\R^2,\R^2)$.
	Then we transform via $\Phi_2 : \R^3 \pfeil \R^3, x \mapsto (x_1 - \beta \circ H^{-1} (x_2,x_3) , x_2 , x_3)$. We easily check that $F := \Phi_2 \Phi_1$ satisfies the desired properties.
	\end{proof}
Define now functions $G_1,G_2,G_3,G_4,G_5$ and $G_6$ via
\begin{equation*}
g_j(x,t) = G_j(t,F(x)), \; j=1,...,5, \quad h_0(x) =G_6(F(x)), \quad  x \in \R^n_\gamma, t \in \R_+, 
\end{equation*}
cf. \cite{johnsendiffeo}.
We also introduce $(\bar h, \bar \mu) := (h,\mu) \circ F$. Then the problem \eqref{fjnjj43t5743} for $(h,\mu)$ is equivalent to the upper half space problem for $(\bar h,\bar \mu)$ reading as
\begin{subequations} \label{fjnjdj43t5743}
	\begin{align} \label{923674937we64z54}
	\partial_t \bar h + \omega^3 \bar h + \ljump n_{\Sigma} \cdot \nabla \bar \mu \rjump &= \mathcal B_1 (\bar \mu) +G_1, && \text{on } \Sigma, \\
	\bar \mu|_{\Sigma} + \Delta_{\Sigma} \bar h &= \mathcal B_2(\bar h) +G_2, && \text{on } \Sigma, \\
	\omega^2 \bar \mu - \Delta_x \bar \mu &= \mathcal B_3(\bar{\mu})+G_3, && \text{on } \R^n_+ \backslash \Sigma,  \label{830xcv4765g347} \\
	e_n \cdot \nabla \bar \mu |_{\partial \R^n} &= \mathcal B_4(\bar{\mu})+ G_4, && \text{on } \partial\R^n_+, \label{83047csdf65347A}\\
	e_n\cdot \nabla_{\Sigma} \bar h |_{\partial\Sigma} &= \mathcal B_5(\bar h) + G_5, && \text{on } \partial\Sigma, \\
	\bar h|_{t=0} &= G_6, && \text{on } \Sigma, \label{9236sd74937d6454Z}
	\end{align}
\end{subequations}
where the perturbation operators are given by
\begin{align}
\mathcal B_1(\bar \mu) &= \ljump n_{\Sigma_\beta} \cdot \nabla (\bar \mu \circ F) \rjump - \ljump ( n_\Sigma \circ F) \cdot \nabla \bar \mu \rjump, \\
\mathcal B_2 (\bar h) &= \Delta_{\Sigma_\beta} (\bar h \circ F) - \Delta_\Sigma \bar h \circ F, \\
\mathcal B_3 (\bar \mu) &= \Delta_x (\bar \mu \circ F) - \Delta \bar\mu \circ F, \\
\mathcal B_4(\bar \mu)&= e_n \cdot (\nabla \bar\mu \circ F) - n_\gamma \cdot \nabla( \bar \mu \circ F), \\
\mathcal B_5(\bar h) &= e_n \cdot \nabla_\Sigma \bar h - n_\gamma \cdot \nabla_{\Sigma_\beta}(\bar h \circ F).
\end{align}
Define $\mathcal B := (\mathcal B_1,\mathcal B_2,\mathcal B_3,\mathcal B_4,\mathcal B_5,0)$. By careful estimates we can now show that the operator norm of $\mathcal B$ is as small as we like in terms of the $\omega$-dependent norms by choosing $\omega >0$ large enough and the time interval and $|\beta|_{C^1} + |\gamma|_{C^1}$ small enough. 

Let 
\begin{equation} \label{9384765364}
\mathbb E(T) := \left( W^1_p(0,T;X_0) \cap L_p(0,T;X_1) \right) \times L_p(0,T;W^2_q(\R^n_+ \backslash \Sigma)
\end{equation}
and
\begin{align}
\label{9384765364B}\mathbb F(T) &:= L_p(0,T;X_0) \times L_p(0,T; W^{2-1/q}_q(\Sigma)) \times L_p(0,T;L_q(\R^n_+)) \times \\ &\times L_p(0,T; W^{1-1/q}_q(\partial\R^n_+)) \times \\
&\times \left(  F^{1-2/(3q)}_{pq} (0,T; L_q (\partial\Sigma)) \cap L_p(0,T; W^{3-2/q}_{q} (\partial\Sigma)) \right) \times X_\gamma.
\end{align}
We equip $\mathbb E(T)$ and $\mathbb F(T)$ with the $\omega$-weighted norms of Section \ref{section34345439}. Then we can show that there is some small $\alpha = \alpha(p)$, such that
\begin{equation} \label{87364537465gfsd}
| \mathcal B|_{\mathcal B(\mathbb E(T) ; \mathbb F(T))} \leq C(\beta,\gamma)(\omega^{-1/q} + \omega^{-1}) + \varepsilon C(\omega,\beta,\gamma,F) + T^\alpha C(\omega,\beta,\gamma)
\end{equation}
for some constants $C(\beta,\gamma), C(\omega,\beta,\gamma, F), C(\omega,\beta,\gamma) > 0$, whenever $|\beta|_{C^1} + |\gamma|_{C^1} \leq \varepsilon$. Note that by first choosing $\omega > 0$ sufficiently large and then $\varepsilon >0$ and $T>0$ sufficiently small, the right hand side gets as small as we like. 

Let now $L_\omega$ be the linear operator defined by the left hand side of \eqref{fjnjdj43t5743}. Then a Neumann series argument shows that $L_\omega + \mathcal B = L_\omega (I + L_\omega^{-1}\mathcal B)$ is invertible between the spaces equipped with the ($\omega$-weighted) norms.
 This way, the following result is obtained.
\begin{theorem} \label{9dfg38645}
Let $\beta, \gamma$ be smooth curves. Then there exists some possibly large $\omega_0>0$, some small $T>0$ and some small $\varepsilon > 0$, such that if $\omega \geq \omega_0$, $|\beta|_{C^1} + |\gamma|_{C^1} \leq \varepsilon$, the system \eqref{fjnjj43t5743} has maximal $L_p-L_q$-regularity. To be more precise, this means that if we replace $\Sigma$ by $\Sigma_\beta$ and $\R^n_+$ by $\R^n_\gamma$ in \eqref{9384765364} and \eqref{9384765364B}, there is for every $(g_1,g_2,g_3,g_4,g_5,h_0) \in \mathbb F(T)$ a unique solution $(h,\mu) \in \mathbb E(T)$ of \eqref{fjnjj43t5743}. Furthermore, $|h|_{E,1,\omega} + |\mu |_{E,2,\omega}$ is bounded by
\begin{equation}
 2M ( |g_1|_{F,1,\omega} +|g_2|_{F,2,\omega}+|g_3|_{F,3,\omega}+|g_4|_{F,4,\omega}+ |g_5|_{F,5,\omega} + |h_0|_{F,6,\omega} ),
\end{equation}
where $M >0$ is as in \eqref{9346858347654} and in particular independent of $\omega$.
\end{theorem}
\subsection{Localization Procedure}
Let us now be concerned with the shifted problem on a bounded smooth domain $\Omega \subset \R^n$, where $\Sigma$ is a perpendicular smooth surface inside.

 More precisely, the system reads as
\begin{subequations} \label{fjnjj43574fgdfg3}
\begin{align} \label{9236749376dfgdfg454}
\partial_t h + \omega^3 h + \ljump n_\Sigma \cdot \nabla \mu \rjump &= g_1, && \text{on } \Sigma, \\
\mu|_\Sigma + \Delta_{\Sigma} h &= g_2, && \text{on } \Sigma, \\
\omega^2 \mu - \Delta \mu &= g_3, && \text{on } \Omega\backslash \Sigma,  \label{8dfgdfg304765347} \\
n_{\partial\Omega} \cdot \nabla \mu |_{\partial \Omega} &= g_4, && \text{on } \partial\Omega, \label{830476dfgdfg5347A}\\
n_{\partial\Omega}  \cdot \nabla_{\Sigma} h |_{\partial\Sigma} &= g_5, && \text{on } \partial\Sigma, \\
h|_{t=0} &= h_0, && \text{on } \Sigma, \label{9236ergrg749376454Z}
\end{align}
\end{subequations}
where $\omega \geq \omega_0$ and $\omega_0 > 0$ is as in Theorem \ref{9dfg38645}.
The main result reads as follows.
\begin{theorem}
Let $n=2,3$, $\Omega \subset \R^n$ be a bounded, smooth domain, $\omega \geq \omega_0$, $6 \leq p < \infty,$ $q \in (3/2,2) \cap (2p/(p+1),2p)$ and $\Sigma$ be a smooth surface inside intersecting $\partial\Omega$ at a constant ninety degree angle.

Then, there is some $T>0$, such that for every $(g_1,g_2,g_3,g_4,g_5,h_0) \in \mathbb F(T)$ satisfying \eqref{8926739873C} there is a unique solution $(h,\mu) \in \mathbb E(T)$ of \eqref{fjnjj43574fgdfg3}.
\end{theorem}
\begin{proof}
Firstly, we can reduce the system to the case where $(g_2,g_3,g_4,h_0) = 0$ by solving auxiliary problems, cf. the proof of Theorem \ref{93453845437} and Theorem \ref{3948srthth6543}. We are now left to solve
\begin{subequations} \label{fjnjj4357D4fgdfg3}
\begin{align} \label{9236749376edfgdfg454}
\partial_t h + \omega^3 h + \ljump n_\Sigma \cdot \nabla \mu \rjump &= g_1, && \text{on } \Sigma, \\
\mu|_\Sigma + \Delta_{\Sigma} h &= 0, && \text{on } \Sigma, \\
\omega^2 \mu - \Delta \mu &= 0, && \text{on } \Omega\backslash \Sigma,  \label{8dfgdfgd304765347} \\
n_{\partial\Omega} \cdot \nabla \mu |_{\partial \Omega} &= 0, && \text{on } \partial\Omega, \label{830476dfgtdfg5347A}\\
n_{\partial\Omega}  \cdot \nabla_{\Sigma} h |_{\partial\Sigma} &= g_5, && \text{on } \partial\Sigma, \\
h|_{t=0} &= 0, && \text{on } \Sigma, \label{9236erdgrg749376454Z}
\end{align}
\end{subequations}
for possibly modified right hand sides which we do not relabel.

We will now show existence and uniqueness of the solution of this system via the localization method. 
To this end let $(\phi_j)_{j = 0,...,N}\subseteq C^\infty_0(\Rn)$ be a smooth partition of unity with respect to $\Omega$ and the open sets $(U_j)_{j = 0,...,N}\subseteq \Rn$, that is, the support of $\phi_j$ is contained in $U_j$ for each $j=0,...,N$ and $\Omega \subseteq \bigcup_{j=0,...,N} U_j.$ Furthermore, let $(\psi_j)_{j =0,...,N} \subseteq C^\infty_0(\Rn)$ be smooth functions with compact support in $U_j$ such that $\psi_j \equiv 1$ on supp $\phi_j$ for every $ 0 \leq j \leq N.$

Now, by choosing $N$ finite but sufficiently large and, corresponding to that, the open sets $U_j$ sufficiently small, we can assume that, up to a rotation, for each $j = 0,...,N$ there exist smooth curves $\gamma_j, \beta_j$ such that
\begin{equation}
U_j \cap \Omega = \R^n_{\gamma_j} \cap \Omega, \quad U_j \cap \Sigma = \R^n_{\gamma_j} \cap \Sigma_{\beta_j}.
\end{equation}
Furthermore, again by a smallness argument, we can choose $\gamma_j$ and $\beta_j$ such that the $C^1$-norm is as small as we like.

We now assume for a moment that we have a solution $(h,\mu)$ of \eqref{fjnjj4357D4fgdfg3} to derive an explicit representation formula. We therefore multiply every equation with $\phi_j$ and get corresponding equations for the localized functions $(h^j,\mu^j) := \phi_j(h,\mu)$. By doing so, we obtain
\begin{subequations}
\begin{align*}
\omega^3 h^j+ \partial_t h^j + \ljump n_\Sigma \cdot \nabla \mu^j \rjump &=   \phi_j g_1 - \mu|_{\Sigma_j} \nabla \phi_j \cdot n_{\Sigma} , & \text{on } \Sigma_{j}, \label{msmodelomegaj} \\
\mu^j + \Delta_\Sigma h^j &= -(\Delta_\Sigma \phi_j)h - 2\textstyle\sum_{lm} g^{lm} \partial_l \phi_j \partial_m h, & \text{on } \Sigma_{j}, \\
\omega^2 \mu^j - \Delta \mu^j &= (\Delta \phi_j )\mu + 2 \nabla \phi_j \cdot \nabla \mu, & \text{on } \Omega_{j} \backslash \Sigma_j ,\\
n_{\partial {\Omega}} \cdot \nabla\mu^j|_{\partial\Omega_j} &= n_{\partial {\Omega}} \cdot \nabla \phi_j \mu|_{\partial\Omega_j},& \text{on } \partial{\Omega}_j ,\\
n_{\partial {\Omega}} \cdot \nabla_\Sigma h^j|_{\partial \Sigma_j} &= \phi_j g_5 + n_{\partial {\Omega}} \cdot \nabla_\Sigma \phi_j h|_{\partial \Sigma_j},  &\text{on } \partial\Sigma_j, \\
h^j|_{t=0} &= 0, & \text{on } \Sigma_j,
\end{align*} \end{subequations}
where $\Sigma_j := \Sigma_{\beta_j}$, $\Omega_j := \R^n_{\gamma_j}$, $(g_{lm})$ is the first fundamental form of $\Sigma_j$ with respect to the surface $\Sigma$ and $(g^{lm})$ its inverse. This way, we obtain a finite number of bent half space problems. Denote by $L^j$ the linear operator on the right hand side of the above system, as well as the data $G^j :=(\phi_j f,0,0,0,\phi_j b,0 )$ and the perturbation operator $R^j$ such that the right hand side equals $G^j + R^j(h,\mu)$, we can write the system of localized equations as
\begin{equation}
L^j (h^j,\mu^j) = G^j + R^j(h,\mu), \quad j=0,...,N.
\end{equation}
Since each $L^j$ is invertible, we may derive the representation formula
\begin{equation} \label{7364fdgfg34}
(h,\mu) = \sum_{j=0}^N \psi_j (L^j)^{-1} G^j + \sum_{j=0}^N \psi_j (L^j)^{-1} R^j(h,\mu).
\end{equation}
Since now $ R :=  \sum_{j=0}^N \psi_j (L^j)^{-1} R^j$ is of lower order, a Neumann series argument now yields that $I - R$ is invertible if $T>0$ is small enough, hence we can rewrite \eqref{7364fdgfg34} as
\begin{equation} \label{93467546398453}
(h,\mu) = (I-R)^{-1} \sum_{j=0}^N \psi_j (L^j)^{-1} G^j.
\end{equation}
Letting $L$ be the linear operator from the left hand side of \eqref{fjnjj4357D4fgdfg3},  we obtain from \eqref{93467546398453} that $L$ is injective, has closed range and a left inverse. It remains to show that $L : \mathbb E(T) \pfeil \mathbb F(T)$ has a right inverse. To this end let $z \in \mathbb F(T)$ be arbitrary. Define \begin{equation} \label{927693dgf4876}
S := (I-R)^{-1} \sum_{j=0}^N \psi_j (L^j)^{-1} \phi_j.
\end{equation} Applying $I-R$ to both sides of \eqref{927693dgf4876} yields a formula for $Sz$, which then entails
\begin{equation}
LSz = z + \sum_{j=0}^N L^j \psi_j (L^j)^{-1} R^j Sz + \sum_{j=0}^N [L^j,\psi_j](L^j)^{-1} \phi_j z, \quad z \in \mathbb F(T).
\end{equation}
Letting $S^R :=  \sum_{j=0}^N L^j \psi_j (L^j)^{-1} R^j S + \sum_{j=0}^N [L^j,\psi_j](L^j)^{-1} \phi_j$, we can show, using again a Neumann series argument involving the fact that the commutator is lower order, that $I + S^R$ is invertible if $T>0$ is small enough. The right inverse of $L$ is therefore given by $S(I+S^R)^{-1}$. This then concludes the proof.
\end{proof}

\subsection{The non-shifted model problem on bounded domains}
In this section we are concerned with problem \eqref{fjnjj43574fgdfg3} for $\omega = 0$. The main result is the following.
\begin{theorem} \label{029786802356805236}
Let $n=2,3$, $\Omega \subset \R^n$ be a bounded, smooth domain, $6 \leq p < \infty,$ $q \in (3/2,2) \cap (2p/(p+1),2p)$ and $\Sigma$ be a smooth submanifold with boundary $\partial\Sigma$ such that $\mathring \Sigma$ is inside $\Omega$ and $\Sigma$ meets $\partial\Omega$ at a constant ninety degree angle.

Then, there is some $T>0$, such that for every $(g_1,g_2,g_3,g_4,g_5,h_0) \in \mathbb F(T)$ satisfying the compatibility condition \eqref{8926739873C} there is a unique solution $(h,\mu) \in \mathbb E(T)$ of \eqref{fjnjj43574fgdfg3} for $\omega = 0$. Furthermore, the solution map is continuous between these spaces.
\end{theorem}
\begin{proof}
As in the previous section we may reduce to the case $(g_2,g_3,g_4,h_0) = 0$. It is also clear that the $\omega^3$-shift in equation \eqref{9236749376dfgdfg454} can easily be resolved to the case $\omega =0$ by an exponential shift in solution and data. We are therefore left to solve
\begin{subequations} \label{39746534987653048756}
\begin{align}
\partial_t h + \ljump n_\Sigma \cdot \nabla T_0 \Delta_\Sigma h \rjump &= g_1, & &\text{on } \Sigma,\\
n_{\partial\Sigma} \cdot \nabla_\Sigma h|_{\partial\Sigma} &= g_5, & &\text{on } \partial\Sigma, \\
h|_{t=0} &= 0, &&\text{on } \Sigma,
\end{align}
\end{subequations}
where $T_0 g$ is the unique solution of the two-phase elliptic problem
\begin{subequations} \label{fghfghhfgh}
	\begin{align}
	-\Delta u &= 0, && \text{in } \Omega \backslash \Sigma, \\
	u|_\Sigma &= g, && \text{on } \Sigma, \\
	n_{\partial\Omega} \cdot \nabla u|_{\partial\Omega} &= 0, && \text{on } \partial\Omega,
	\end{align}
\end{subequations}
cf. 
 Appendix \ref{9348579348564305}. Also from Appendix \ref{9348579348564305} we obtain that $T_0 \Delta_\Sigma h = T_\eta \Delta_\Sigma h + \eta (\eta-\Delta)^{-1}T_0 \Delta_\Sigma h$, for all $\eta \geq \eta_0$. This implies that problem \eqref{39746534987653048756} is equivalent to
\begin{subequations} \label{39746534ff987653048756}
\begin{align}
\partial_t h + \ljump n_\Sigma \cdot \nabla T_\eta \Delta_\Sigma h \rjump &= g_1+\eta\ljump n_\Sigma \cdot \nabla (\eta-\Delta)^{-1}T_0 \Delta_\Sigma h \rjump , && \text{on } \Sigma,\\
n_{\partial\Sigma} \cdot \nabla_\Sigma h|_{\partial\Sigma} &= g_5, && \text{on } \partial\Sigma, \\
h|_{t=0} &= 0, &&\text{on } \Sigma,
\end{align}
\end{subequations}
provided $\eta \geq \eta_0$. Now choose large enough $\eta$ to render the left hand side of \eqref{39746534ff987653048756} to be an invertible operator. We may estimate
\begin{equation} \label{29865976430578}
| \ljump n_\Sigma \cdot \nabla (\eta-\Delta)^{-1}T_0 \Delta_\Sigma h \rjump |_{L_p(0,T;X_0)} \leq C(\eta) T^{1/(3p)} |h|_{W^1_p(0,T;X_0) \cap L_p(0,T;X_1)},
\end{equation}
whence choosing $T>0$ sufficiently small
and a standard Neumann series argument complete the proof. Here, \eqref{29865976430578} stems from the solution formula \eqref{9280634097236}, real interpolation method and H\"older inequality.
\end{proof}

\section{Nonlinear Well-Posedness}
In this section we will show local well-posedness for the full nonlinear (transformed) system \eqref{fullnonlienar8765} and therefore obtain that also the system \eqref{87959649764dfs} is well-posed.
We will use the maximal $L_p-L_q$ regularity result for the underlying linear problem and a contraction argument via the Banach's fixed point principle.

The main result reads as follows.
\begin{theorem} \label{83457639457628g}
Let $6 \leq p < \infty$, $q \in (5/3,2) \cap (2p/(p+1),2p)$ and $h_0 \in X_\gamma$ sufficiently small. Then there is some possibly small $\tau > 0$, such that \eqref{fullnonlienar8765} has a unique strong solution on $(0,\tau)$, that is, there are 
\begin{equation}
h \in W^1_p(0,\tau; X_0) \cap L_p(0,\tau;X_1), \quad \mu \in L_p(0,\tau;W^2_q(\Omega \backslash \Sigma)),
\end{equation}
solving \eqref{fullnonlienar8765} on $(0,\tau)$, whenever $h_0 $ satisfies the initial compatibility condition $\partial_\nu h_0 = 0$ on $\partial\Sigma$.
\end{theorem}

\begin{proof}
Define $L : \mathbb E(T) \pfeil \mathbb F(T)$ as
\begin{equation}
L(h,\mu) = \begin{pmatrix}
\partial_t h + \ljump n_\Sigma \cdot \nabla \mu \rjump \\
\mu|_\Sigma - P(0)h \\
\Delta \mu \\
n_{\partial\Omega }\cdot \nabla \mu|_{\partial\Omega} \\
n_{\partial\Sigma} \cdot \nabla_\Sigma h|_{\partial \Sigma} \\
h|_{t=0}
\end{pmatrix}.
\end{equation}
We now reduce to trivial initial data as follows, cf. \cite{koehnepruesswilke}. Since $h_0$ satisfies the compatibility condition, we may solve
\begin{equation}
Lz_* = L(h_*,\mu_*) = (0,0,0,0,0,h_0)
\end{equation}
by some $z_* = (h_*,\mu_*) \in \mathbb E(T)$. Then the problem \eqref{fullnonlienar8765}  is equivalent to finding some $z = (h,\mu) \in {_0 \mathbb E}(T)$ solving
\begin{equation}
L(z)= N(z+z_*)-Lz_* =: \tilde N(z), \quad \text{in } {_0 \mathbb F(T)},
\end{equation}
where the nonlinear part is given by
\begin{align}
&N(z+z_*) = \\ & \begin{pmatrix}
\ljump n_\Sigma^{h+h_*} \cdot \nabla_{h+h_*} (\mu+\mu_*) \rjump - \ljump n_\Sigma \cdot \nabla (\mu+\mu_*) \rjump + (\beta(h+h_*)|n_\Sigma^{h+h_*} - n_\Sigma ) \\
K(h+h_*) - P(0)(h+h_*) \\
(\Delta-\Delta_{h+h_*})(\mu+\mu_*) \\
n_{\partial\Omega} \cdot \nabla (\mu+\mu_*)|_{\partial\Omega} - n_{\partial\Omega}^{h+h_*} \cdot \nabla_{h+h_*} (\mu+\mu_*)|_{\partial\Omega} \\
n_{\partial\Sigma} \cdot \nabla_\Sigma (h+h_*)|_{\partial\Sigma} - n_{\partial\Sigma}^{h+h_*} \cdot n_\Sigma^{h+h_*} \\
h_0
\end{pmatrix}. \nonumber
\end{align}
We may now define $K : {_0 \mathbb E(T)} \pfeil  {_0 \mathbb E(T)}$ by $z \mapsto L^{-1}\tilde N(z) = L^{-1}( N(z+z_*) - Lz_* )$. By restricting to functions with vanishing trace at time zero, we get that the operator norm $| L^{-1}|_{ \mathcal B ( {_0 \mathbb F(T) }; {_0 \mathbb E(T)} )}$ stays bounded as $T \pfeil 0$. 

\begin{lemma} \label{283746504376}
The mapping $N : { \mathbb E(T)} \pfeil { \mathbb F(T)}$ is well-defined and bounded. Furthermore, $N \in C^2( { \mathbb E(T)}; { \mathbb F(T)})$.
Furthermore, $N$ allows for contraction estimates in a neighbourhood of zero, that is,
\begin{equation*}
| N(z_1 + z_*) - N(z_2 + z_*) |_{ _0 \mathbb F(T)} \leq C(| z_1 |_{ _0 \mathbb E(T)} + |  z_2|_{ _0 \mathbb E(T)} + |  z_*|_{ \mathbb E(T)}) | z_1 - z_2|_{ _0 \mathbb E(T)},
\end{equation*}
for all $z_1, z_2 \in \mathsf B(r;0) \subset {_0 \mathbb E(T)}$, if $r>0$ and $T=T(r) >0$ are sufficiently small. Here, $\mathsf B(r;0)$ denotes the closed ball around $0$ with radius $r > 0$.
\end{lemma}
Let now $\delta > 0$, such that $|h_0|_{X_\gamma} \leq \delta.$
By choosing $r>0$, $T=T(r) >0$ and $\delta = \delta(T) >0$ sufficiently small, we ensure $K$ to be a $1/2$-contraction on $\mathsf B(r,0) \subset  {_0 \mathbb E(T)}$. Note at this point that $|  z_*|_{ \mathbb E(T)} \leq C(T) \delta$.
Let us note that
\begin{equation}
\tilde N(0) = N(z_*) - Lz_*, \quad K(0) = L^{-1} \tilde N(0).
\end{equation}
Note $\tilde N(0) \in {_0 \mathbb F(T)}$, whence
\begin{equation}
| K(0) |_{ _0 \mathbb E(T)} \leq |L^{-1} |_{ \mathcal B({_0 \mathbb F(T)} ; {_0 \mathbb E(T)})} | \tilde N(0)|_{ _0 \mathbb F(T)}.
\end{equation}
Now we note that $\tilde N(0)$ is quadratic in $z_* = (h_*, \mu_*)$ except for the term $Q(0)$ in $\tilde N(0)_2$. Using 
\begin{equation}
| Q(0) |_{L_p(0,T; W^{2-1/q}_q(\Sigma))} \leq T^{1/p} | Q(0) |_{L_\infty(0,T; W^{2-1/q}_q(\Sigma))} \pfeil_{ T \pfeil 0}  0,
\end{equation}
as well as
\begin{equation}
| z_* |_{ \mathbb E(T)} \leq |L^{-1} |_{ \mathcal B({ \mathbb F(T)} ; {\mathbb E(T)})} |h_0|_{X_\gamma}
\end{equation}
finish the proof by choosing first $r>0, T = T(r) >0$ and then $|h_0|_{X_\gamma}$ small enough.
Hence Banach's fixed point principle yields the existence of a unique fixed point $\bar z \in \mathsf B(r,0) \subset {_0 \mathbb E(T)}$.
By standard arguments this is then the unique fixed point in all of ${_0 \mathbb E(T)}$. Assume there is a different fixed point $\tilde z$ in a possibly larger ball $\mathsf B(r',0)$. Then define
\begin{equation} \label{34853454354355454}
T_* := \sup \{ t  \geq 0 : \bar z (t') = \tilde z (t') \text{ for all } 0 \leq t' \leq t \}.
\end{equation}
By performing the above fixed point argument on the larger ball $\mathsf B(r',0)$ on a smaller time interval we see that $T_* > 0$. Then solve the nonlinear problem with initial value $\bar z (T_*) \in X_\gamma$. Note that $\bar z (T_*)$ satisfies the compatibility condition $\partial_\nu [ \bar z (T_*)] = 0$ on $\partial \Sigma$, hence we may obtain a unique solution on a larger time interval $(0, T_* + \epsilon_* )$ for some $\epsilon_* > 0$. This contradicts \eqref{34853454354355454} and the fixed point has to be unique.
This then in turn yields the uniqueness of the solution to \eqref{fullnonlienar8765}.
\end{proof}
Let us comment on how to prove Lemma \ref{283746504376}. Using the differentiability properties from Lemma \ref{93845703746534}, the statement easily follows for the components $N_1,N_3$ and $N_4$. The  decomposition $K(h)=P(h)h + Q(h)$ from Lemma \ref{9286fg349840} renders a proof for $N_3$. For $N_5$ we note that Depner in \cite{depnerdiss} calculated the linearization of the ninety-degree angle boundary condition \eqref{2938dfg623}, which turns out to be $n_{\partial\Sigma} \cdot \nabla_\Sigma h|_{\partial\Sigma} = 0$, which in turn then allows for estimates for $N_5$. Note that we also use the Banach algebra property for the trace space, cf. Theorem \ref{banachalgebra} in the appendix. For sake of readability we omit the details here.

\begin{remark}
We point out that the proof of Theorem \ref{83457639457628g} also gives well-posedness of \eqref{fullnonlienar8765} in the case where $\Omega = G \times (L_1,L_2)$ is a bounded container in $\R^n$, $n=2,3$. Hereby $G \subset \R^{n-1}$ is a smooth, bounded domain. In this case there is another model problem in the localization procedure for the linear problem stemming from when the top and bottom of the container $G \times \{ L_1 , L_2 \}$ intersect the walls $\partial G \times (L_1,L_2)$. This elliptic problem, although being a problem on a domain with corners, admits full regularity for the solution, cf. the appendix in Section \ref{384075647564574695}.
\end{remark}

\section{Convergence to equilibria}
This section is devoted to the long-time behaviour of solutions to \eqref{87959649764dfs} starting close to
equilibria. We will characterize the set of equilibria, study the spectrum of the linearization of the transformed Mullins–Sekerka equations around an equilibrium and apply the generalized principle of linearized stability to show that solutions starting sufficiently close to certain equilibria converge to an equilibirum at an exponential rate in $X_\gamma$.

We note that the potential $\mu$ can always be reconstructed by $\Gamma(t)$ by solving the elliptic two-phase problem
\begin{subequations}
\begin{align}
\mu|_{\Gamma(t)} &= H_{\Gamma(t)}, && \text{on } \Gamma(t), \label{82734659645} \\
\Delta \mu &= 0, && \text{in } \Omega \backslash \Gamma(t), \\
n_{\partial\Omega} \cdot \nabla \mu|_{\partial\Omega} &= 0, && \text{on } \partial\Omega.
\end{align}
\end{subequations}
Whence we may concentrate on the set of equilibria for $\Gamma (t)$.

It now can easily be shown that for a stationary solution $\Gamma$ of \eqref{87959649764dfs} with $V_\Gamma = 0$ the corresponding chemical potential $\mu$ is constant, since then $\mu$ and $\nabla \mu$ have no jump across the interface $\Gamma$ and $\mu \in W^2_q(\Omega)$ solves a homogeneous Neumann problem on $\Omega$. By \eqref{82734659645}, the mean curvature $H_\Gamma$ is constant. The set of equilibria for the flow $\Gamma(t)$ is therefore given by
\begin{equation}
\mathcal E = \{  \Gamma : H_\Gamma = \text{const.} \}.
\end{equation}
Let us now consider the case where $\Omega \subset \R^n$, $n=2,3$, is a bounded container, that is,  $\Omega := \Sigma \times (L_1,L_2),$ where $-\infty < L_1 < 0 < L_2 < \infty$ and $\Sigma \subset \R^{n-1} \times \{ 0 \}$ is a bounded domain and $\partial\Sigma$ is smooth.

Note that flat interfaces are equilibria. Arcs of circles intersecting $\partial\Omega$ at a ninety degree angle also belong to $\mathcal E$, since then \eqref{983457034} is also satisfied. 

 If we now additionally assume that the contact points between $\Gamma$ and $\partial\Omega$ are only on the walls of the cylinder and $\Gamma$ is given as a graph over $\Sigma$, we may even deduce that $H_\Gamma = 0$, that is, $\Gamma$ is a flat interface described by a constant height function over the reference surface. This follows from the fact that we can describe $\Gamma$ as graph of a height function $h$ over $\Sigma$. Then using the well-known formula
 \begin{equation} \label{9485674rthrthrthtrh}
 H_\Gamma =  \div \bigg( \frac{\nabla h}{\sqrt{1+|\nabla h|^2}} \bigg)
\end{equation}  
 and the boundary condition \eqref{983457034} on $\partial\Gamma$ renders $H_\Gamma = 0$. 
 Indeed, assume that $\Gamma = \Gamma_h$ is a graph of $h$ over $\Sigma$. We may assume that $h$ has mean value zero and we already know $H_\Gamma$ is constant, but may be nonzero. Then an integration by parts entails
 \begin{equation}
 0 = \int_\Sigma h H_\Gamma dx = - \int_\Sigma \frac{ \nabla h \cdot \nabla h}{ \sqrt{ 1 + | \nabla h|^2} } dx.
 \end{equation}
 The boundary integral vanishes due to \eqref{983457034} and renders $\nabla h$ to be zero in $\Sigma$, hence $h$ is constant. This implies $H_\Gamma = 0$.
 
 We will now study the problem for the height function in an $L_p$-setting.
 We now rewrite the geometric problem \eqref{983457034} as an abstract evolution equation for the height function $h$, cf. \cite{abelswilke}, \cite{pruessmulsek}, \cite{eschersimonett}. As seen before, by means of Hanzawa transform, the full system reads as
\begin{subequations} \label{fullnrsthrthonlienar8765}
\begin{align}
\partial_t h &= - \ljump n_\Sigma^h \cdot \nabla_h \mu \rjump , && \text{on } \Sigma, \label{456rthrth356} \\
\mu|_{\Sigma} &= H_\Gamma(h), && \text{on } \Sigma, \\
\Delta_h \mu &= 0, && \text{in } \Omega \backslash \Sigma, \\
n_{\partial\Omega} \cdot \nabla_h \mu|_{\partial\Omega} &= 0, && \text{on } \partial\Omega, \\
n_{\partial\Sigma} \cdot \nabla_\Sigma h|_{\partial\Sigma} &= 0, && \text{on } \partial\Sigma, \label{293fgdhfgh8dfg623} \\
h|_{t=0} &= h_0, && \text{on } \Sigma. \label{45fghfg6356Z}
\end{align}
\end{subequations}
Let us note that due to working in a container, the highly nonlinear angle condition \eqref{2938dfg623} reduces to a linear one, condition \eqref{293fgdhfgh8dfg623}. Define $B(h)g := \ljump n_\Sigma^h \cdot \nabla g \rjump$ and $S(h)g$ as the unique solution of the elliptic problem
\begin{subequations}
\begin{align}
\mu|_{\Sigma} &= g, && \text{on } \Sigma, \\
\Delta_h \mu &= 0, && \text{in } \Omega \backslash \Sigma, \\
n_{\partial\Omega} \cdot \nabla_h \mu|_{\partial\Omega} &= 0, && \text{on } \partial\Omega.
\end{align}
\end{subequations}
Recalling Lemma \ref{9286fg349840}, we may rewrite \eqref{fullnrsthrthonlienar8765} as an abstract evolution equation,
\begin{subequations} \label{930476853486geudfgh}
\begin{align}
\frac{d}{dt}h(t) + A(h(t))h(t) &= F(h(t)), \quad t \in \R_+, \\
h(0) &= h_0,
\end{align}
\end{subequations}
where $A(h)g := B(h)S(h)P(h)g$, equipped with domain
\begin{equation}
D(A(h)) := W^{4-1/q}_q(\Sigma) \cap \{ g : n_{\partial\Sigma} \cdot \nabla g = 0 \text{ on } \partial\Sigma \},
\end{equation}
and $F(g) := -B(g)S(g)Q(g)$. We now want to study \eqref{930476853486geudfgh} in an $L_p$-setting. Define
\begin{equation}
X_0 := W^{1-1/q}_q(\Sigma), \quad X_1 := W^{4-1/q}_q(\Sigma), \quad X_\gamma := (X_1,X_0)_{1-1/p,p}.
\end{equation}
We now interpret problem \eqref{930476853486geudfgh} as an evolution equation in $L_p(\R_+;X_0)$, fitting in the setting of Prüss, Simonett and Zacher \cite{pruessmulsek}.

Regarding the linearization of \eqref{87959649764dfs}, we have the following result.
\begin{lemma} \label{98376458037645083476}
Let $6 \leq p < \infty$, $q \in (5/3,2) \cap (2p/(p+1),2p)$. Then the following statements are true.
\begin{enumerate} \itemsep-0.1em
\item The derivative of $H_\Gamma$ at $h=0$ is given by $ [ h \mapsto \Delta_\Sigma h]$, 
\item there is an open neighbourhood of zero $V \subset X_\gamma$, such that $(A,F) \in C^1(V; \mathcal B(X_1;X_0) \times X_0)$,
\item the linearization of $A$ at zero is given by $A_0 = A(0)$, where $A_0 : D(A_0) \pfeil X_0$, $A_0 h = -\ljump n_\Sigma \cdot T \Delta_\Sigma h \rjump$ with domain
\begin{equation}
D(A_0) = X_1 \cap \{   h : n_{\partial\Sigma} \cdot \nabla_\Sigma h|_{\partial\Sigma} = 0 \text{ on } \partial\Sigma \}.
\end{equation}
Here, $T : W^{2-1/q}_q(\Sigma) \pfeil W^2_q(\Omega \backslash \Sigma), g \mapsto \chi,$ is the solution operator for the elliptic two-phase problem
\begin{subequations} \label{87tzfc3645836}
	\begin{align}
	\Delta \chi &= 0, && \text{in } \Omega \backslash \Sigma, \\
	\chi|_\Sigma &= g, && \text{on } \Sigma, \\
	n_{\partial\Omega} \cdot \nabla \chi|_{\partial\Omega} &= 0, && \text{on } \partial\Omega,
	\end{align}
\end{subequations}
\item the set of equilibria, that is, the solutions of $A(h)h=F(h)$ is $\mathcal E = \{ h = \text{const}. \}$,
\item $A_0$ has maximal $L_p$-regularity,
\item the kernel of $A_0$ are the constant functions, $N(A_0) = \{ h = \text{const}. \}$,
\item $N(A_0) = N(A_0^2)$.
\end{enumerate}
\end{lemma}
\begin{proof}
\begin{enumerate}\itemsep-0.1em
\item This stems from linearizing \eqref{9485674rthrthrthtrh} at $h= 0$.
\item Again by Lemma \ref{9286fg349840}, there is a small neighbourhood of zero $V\subset X_\gamma$ such that $P  \in C^1(V;\mathcal B(X_1;W^{2-1/q}_q(\Sigma))$ and $Q \in C^1(V;W^{2-1/q}_q(\Sigma))$. Following the lines of \cite{abelswilke} using Lemma \ref{93845703746534} we can show that 
\begin{equation}
S \in C^1(V; \mathcal B( W^{2-1/q}_q(\Sigma) ; W^2_q(\Omega \backslash \Sigma) ) ). 
\end{equation}
Regarding $B$ we note that $X_\gamma \into C^2(\Sigma)$, whence 
\begin{equation}
B \in C^1(V; \mathcal B( W^2_q (\Omega \backslash \Sigma) ; W^{1-1/q}_q(\Sigma) ) ). 
\end{equation}
This shows that $(A,F) \in C^1(V; \mathcal B(X_1;X_0) \times X_0)$.
\item This stems from the fact that $A_0 = A(0)$ and Lemma \ref{9286fg349840}.
\item Let $h \in D(A)$ satisfy $A(h)h = F(h)$. It then follows that $B(h)S(h)H_\Gamma(h) =0$ on $\Sigma$, that is,
\begin{equation}
\ljump n_\Sigma^h \cdot \nabla_h [S(h) H_\Gamma(h)] \rjump = 0, \quad \text{on } \Sigma.
\end{equation}
Note that $S(h)H_\Gamma(h)$ is the unique solution of an $h$-perturbed elliptic problem with homogeneous Neumann boundary conditions. Therefore $S(h)H_\Gamma(h)$ has to be constant. Since $S(h)H_\Gamma(h)$ equals $H_\Gamma(h)$ on $\Sigma$, also $H_\Gamma(h)$ is constant.
We then obtain that the mean curvature $H_\Gamma$ of the interface given as a graph of $h$ over $\Sigma$ is constant. Due to \eqref{293fgdhfgh8dfg623} we may even deduce using formula \eqref{9485674rthrthrthtrh} that $H_\Gamma = 0$. Then $h$ has to be constant.

\item This stems from Theorem \ref{029786802356805236}.
\item Clearly, every constant is an element of $N(A_0)$.  For the converse, let $h \in D(A_0)$, such that $A_0 h=0$. Hence $\chi = T\Delta_\Sigma h$ is constant and thererfore $\Delta_\Sigma h$ is constant. Since $h \in D(A_0)$, an integration by parts shows $\Delta_\Sigma h = 0$. Again since $h \in D(A_0)$, $h$ has to be constant.
\item We only need to show $N(A_0^2) \subset N(A_0)$. Pick some $h \in N(A_0^2)$. Then $A_0 h \in D(A_0) \cap N(A_0)$. Hence $A_0 h$ is constant. Also, $A_0 h$ is in the range of $A_0$. Since every element in the range of $A_0$ has mean value zero, it follows that $A_0 h = 0$, whence $h \in N(A_0)$.
\end{enumerate}
The proof is complete.
\end{proof}
The following theorem enables us to apply the generalized principle of linearized stability of Prüss, Simonett, and Zacher \cite{pruessmulsek} to the evolution equation \eqref{930476853486geudfgh}.
\begin{theorem}
Let $6 \leq p < \infty, q \in (5/3,2) \cap (2p/(p+1),2p)$. Then the trivial equilibrium $h_* = 0$ is normally stable.

More precisely:
\begin{enumerate}\itemsep-0.1em
\item Near $h_* = 0$ the set of equilibria $\mathcal E$ is a $C^1$-manifold in $X_1$ of dimension one.
\item The tangent space of $\mathcal E$ at $h_* = 0$ is given by the kernel of the linearization, $T_0 \mathcal E = N(A_0)$.
\item Zero is a semi-simple eigenvalue of $A_0$, i.e. $X_0 = N(A_0) \oplus R(A_0)$.
\item The spectrum $\sigma(A_0)$ satisfies $\sigma(A_0) \backslash \{ 0 \} \subset \mathbb C_+ := \{ z \in \mathbb C : \operatorname{Re} z > 0 \}$.
\end{enumerate}
\end{theorem}
\begin{proof}
\begin{enumerate}\itemsep-0.1em
\item Around $h_*$, the set of equilibria only consists of constant functions, hence is a one-dimensional linear subspace of $X_1$.
\item This stems from Lemma \ref{98376458037645083476}.
\item Since $D(A_0)$ compactly embeds into $W^{1-1/q}(\Sigma)$, the operator $A_0$ has a compact resolvent and the spectrum $\sigma(A_0)$ only consists of eigenvalues, cf. \cite{engelnagel}. Furthermore, every spectral value in $\sigma(A_0)$ is a pole of finite algebraic multiplicity. By using $N(A_0) = N(A_0^2)$ and Proposition A.2.2 and Remark A.2.4 in \cite{lunardioptimal} we may conclude that the range of $A_0$ is closed in $X_0$ and that there is a spectral decomposition $X_0 = N(A_0) \oplus R(A_0)$. Hence $\lambda = 0$ is semi-simple.
\item Pick $\lambda \in \sigma(A_0)$ with corresponding eigenfunction $h \in D(A_0)$, in other words
\begin{equation} \label{923674583746874df}
\lambda h - A_0 h = 0, \quad \text{in } X_0.
\end{equation}
Testing with $h$ and integrating by parts using $W^{1-1/q}_q(\Sigma) \into L_2(\Sigma)$ yields that
\begin{equation}
0 = |\nabla \chi|_{L_2(\Omega)}^2 + (A_0 h | \Delta_\Sigma h)_{L_2(\Sigma)},
\end{equation}
where $\chi = T \Delta_\Sigma h$.
Testing again the resolvent equation \eqref{923674583746874df} now with $\Delta_\Sigma h$ finally yields
\begin{equation}
\lambda | \nabla_\Sigma h|_{L_2(\Sigma)}^2 = | \nabla \chi |_{L_2(\Omega)}^2.
\end{equation} 
This shows that $\lambda$ is real and $\lambda \geq 0$. In particular, $\sigma(A_0) \backslash \{ 0 \} \subset (0,\infty)$.
\end{enumerate}
Hence $h_*$ is normally stable.
\end{proof}
The following theorem is the main result on stability of solutions. It is an application of the generalized principle of linearized stability of Prüss, Simonett, and Zacher \cite{pruessmulsek} to the evolution equation \eqref{930476853486geudfgh}.

\begin{theorem} \label{thmgeome34954351} 
The trivial equilibrium $h_* = 0$ is stable in $X_\gamma$, and there is some $\delta > 0$ such that the evolution equation
\begin{equation}
\frac{d}{dt}h(t) + A(h(t))h(t) = 0, \quad h(0) = h_0,
\end{equation}
with initial value $h_0 \in X_\gamma$ satisfying $| h_0 - h_* |_{X_\gamma} \leq \delta$ has a unique global in-time solution on $\R_+$,
\begin{equation}
h \in W^1_p(\R_+ ; X_0)   \cap L_p(\R_+;D(A_0)) ,
\end{equation}
which converges at an exponential rate in $X_\gamma$ to some $h_\infty \in \mathcal E$ as $t \pfeil +\infty$.
\end{theorem}
\begin{theorem}(Geometrical version) \label{thmgeome3495435}
Suppose that the initial surface $\Sigma_0$ is given as a graph, $\Sigma_0 = \{ (x,h_0(x)) : x \in \Sigma \}$ for some function $h_0 \in X_\gamma$. Then, for each $\varepsilon > 0$ there is some $\delta(\varepsilon) > 0$, such that if the initial value $h_0 \in X_\gamma$ satisfies $|h_0 - h_*|_{X_\gamma} \leq \delta(\varepsilon)$ for some constant function $h_*$, there exists a global-in-time strong solution $h$ on $\R_+$ of the evolution equation, precisely $h \in L_p(\R_+;D(A_0)) \cap W^1_p(\R_+;X_0))$, and it satisfies $|h(t)|_{X_\gamma} \leq \varepsilon$ for all $t \geq 0$.

Moreover, there is some constant $h_\infty$, such that $\Sigma_{h(t)} \pfeil \Sigma_{h_\infty}$ in the sense of $h(t) \pfeil h_\infty$ in $X_\gamma$ and the convergence is at an exponential rate.
\end{theorem}

Note that by the following theorem we can characterize the limit. It is a priori not clear to which equilibrium the solution converges by the generalized principle of linearized stability.
\begin{theorem}
The limit $h_\infty$ from above has the same mean value as $h_0,$ in other words, $ \int_\Sigma h_0 dx /|\Sigma| =  h_\infty$ .
\end{theorem}
\begin{proof}
The theorem is a consequence of the fact that the Mullins-Sekerka system conserves the measure of the domains separated by the interface in time.
Hence the solution $h$ from Theorem \ref{thmgeome34954351} satisfies
\begin{equation}
\frac{d}{dt}\int_\Sigma h(t,x) dx = 0.
\end{equation}
In particular, 
\begin{equation}
\int_\Sigma h(t,x)dx = \int_\Sigma h_0 (x) dx, \quad t \in \R_+.
\end{equation}
Since $h(t) \pfeil h_\infty$ as $t \pfeil \infty$ in $X_\gamma \into L_1(\Sigma)$, we get
the result.
\end{proof}
\begin{appendices}
\section{Auxiliary problems of elliptic type} \label{9348579348564305}
\subsection{Smooth domains.} Let $\Omega \subset \R^n $ be a bounded domain with smooth boundary $\partial\Omega$. Furthermore let $\Sigma$ be a smooth submanifold of $\R^n$ with boundary such that the interior $\mathring \Sigma$ is inside $\Omega$ and meets $\partial\Omega$ at a constant ninety degree angle.

In this chapter we are concerned with problems of elliptic type,
namely, 
\begin{subequations} \label{293876498765}
	\begin{align}
	(\eta-\Delta) u &= f, && \text{in } \Omega \backslash \Sigma, \\
	u|_\Sigma &= g_1, && \text{on } \Sigma, \\
	n_{\partial\Omega} \cdot \nabla u|_{\partial\Omega} &= g_2, && \text{on } \partial\Omega,
	\end{align}
\end{subequations}
where $\eta > 0$ is a fixed shift parameter, as well as the non-shifted version,
\begin{subequations} \label{873645836}
	\begin{align}
	-\Delta u &= f, && \text{in } \Omega \backslash \Sigma, \\
	u|_\Sigma &= g_1, && \text{on } \Sigma, \\
	n_{\partial\Omega} \cdot \nabla u|_{\partial\Omega} &= g_2, && \text{on } \partial\Omega.
	\end{align}
\end{subequations}
We will show optimal solvability of this problem via a localization method. To this end we consider first the model problem of \eqref{293876498765} on $\R^n_+$ with flat interface $ \{ x_n > 0, x_1 = 0 \}$. 
\begin{theorem}
	Let $\eta > 0, q \in (3/2,2)$ and $\Sigma := \{ x_n > 0, x_1 = 0 \}$. Then, for every $f \in L_q(\R^n _+)$, $g_1 \in W^{2-1/q}_q(\Sigma)$ and $g_2 \in W^{1-1/q}_q(\partial \R_+)$ there exists a unique solution $u \in W^2_q(\R^n_+ \backslash \Sigma)$ of \eqref{293876498765} with $\R^n _+$ replacing $\Omega$.

	 Furthermore, there is some $C(\eta) >0$ and some $K>0$ independent of $\eta$, such that
	 \begin{align}
	 |u&|_{L_q(\R^n_+)} + \eta^{-1/2}|Du|_{L_q(\R^n_+ \backslash\Sigma)} + \eta^{-1}|D^2 u |_{L_q (\R^n_+ \backslash \Sigma)} \\
	 &\leq K \eta^{-1} |f|_{L_q(\R  ^n _+)} + C(\eta)|g_1|_{W^{2-1/q}_q(\Sigma)} + K \eta^{ -1/(2q) - 1/2} |g_2|_{W^{1-1/q}_q(\partial\R^n _+)}. \nonumber
	 \end{align}
\end{theorem}
\begin{proof}
	We first solve an auxiliary upper half space problem to reduce the problem to
	\begin{subequations} \label{2s93876498765}
		\begin{align}
		(\eta-\Delta) u &= 0, && \text{in } \R^n _+ \backslash \Sigma, \\
		u|_\Sigma &= g_1, && \text{on } \Sigma, \\
		\partial_n u|_{\partial\R^n _+} &= 0, && \text{on } \partial\R^n _+,
		\end{align}
	\end{subequations}
	for possibly modified $g_1$ not to be relabeled. Since $\partial_n u = 0$ on the boundary, we may reflect the problem via an even reflection to obtain an elliptic problem on $\dot \R \times \R^{n-1}$. By Theorem \ref{456745645} using $q < 2$ we obtain that $Rg_1 \in W^{2-1/q}_q(\tilde{\Sigma})$, where $\tilde \Sigma := \{ x_1 = 0\}.$ Here, $R$ denotes the aforementioned even reflection in $x_n$-direction. The problem we are left to solve is now
	\begin{subequations} \label{dfhfk48345}
	\begin{align}
	(\eta-\Delta)v &= 0, && \text{in } \R^n \backslash \tilde \Sigma, \\
	v|_{\tilde \Sigma} &= Rg_1, && \text{on } \tilde \Sigma.
	\end{align}
	\end{subequations}
 Let $x' := (x_2,...,x_n)$.	It is now well known that the operator $(\eta-\Delta_{x'})^{1/2}$ with domain $W^1_q(\R^{n-1})$ has maximal regularity on the half line $\R_+$ with respect to the base space $L_q(\R^{n-1})$ and the induced semigroup is analytic. Note that by real interpolation method,
	\begin{equation}
	\left( D((\eta-\Delta_{x'})^{1/2}) , L_q(\R^{n-1}) \right)_{1-1/q,q} = W^{1-1/q}_q(\R^{n-1}),
	\end{equation}
	whence we may solve \eqref{dfhfk48345} by
	\begin{equation} \label{9280634097236}
	v(x,x') =  e^{-(\eta-\Delta_{x'})^{1/2}|x_1|}Rg_1(x'), \quad x_1 \in \R, x' \in \R^{n-1}.
	\end{equation}
	We obtain
	\begin{equation}
	| v|_{\R^n_+} |_{W^1_q (\R^n_+)} \leq C |g_1|_{W^{1-1/q}_q(\Sigma)}.
	\end{equation}
	To obtain the dependence of the shift parameter one proceeds by a scaling argument as in Section \ref{section34345439}. The proof is complete.
	\end{proof}
	By a standard localization argument we can now show that the shifted problem is solvable in the case of a bounded, smooth domain.
	\begin{theorem}
		Let $q \in (3/2,2)$, $\Omega \subset \R^n$ a bounded, smooth domain and $\Sigma$ a smooth surface inside $\Omega$ intersecting the boundary $\partial \Omega$ at a nintey degree angle. Then there is some $\eta_0 \geq 0$, such that if $\eta \geq \eta_0$, for every $(f,g_1,g_2) \in L_q(\Omega) \times W^{2-1/q}_q(\Sigma) \times W^{1-1/q}_q(\partial\Omega)$ there is unique $u \in W^2_q(\Omega \backslash \Sigma)$ solving \eqref{293876498765}. Furthermore, the solution map $(f,g_1,g_2) \mapsto u$ is continuous between the above spaces.
	\end{theorem}
	We will now concern solvability of the non-shifted problem \eqref{873645836}.
	\begin{theorem} \label{3948srthth6543}
		Let $q \in (3/2,2)$. For every $(f,g_1,g_2) \in L_q(\Omega) \times W^{2-1/q}_q(\Sigma) \times W^{1-1/q}_q(\partial\Omega)$ there is unique $u \in W^2_q(\Omega \backslash \Sigma)$ solving \eqref{873645836}. Furthermore, there is some constant $C>0$, such that
		\begin{equation}
		|u|_{W^2_q(\Omega \backslash \Sigma)} \leq C \left( |f|_{L_q(\Omega)} + |g_1|_{W^{2-1/q}_q(\Sigma)} + |g_2|_{W^{1-1/q}_q(\partial\Omega)} \right).
		\end{equation}
	\end{theorem}
	\begin{proof}
		First we choose $\eta > 0$ large enough and solve \eqref{293876498765} by a function $v \in W^2_q(\Omega \backslash \Sigma)$. It therefore remains to solve
		\begin{subequations} \label{873ddd645836}
			\begin{align}
			-\Delta w &= -\eta v, && \text{in } \Omega \backslash \Sigma, \\
			u|_\Sigma &= 0, && \text{on } \Sigma, \\
			n_{\partial\Omega} \cdot \nabla u|_{\partial\Omega} &= 0, && \text{on } \partial\Omega,
			\end{align}
		\end{subequations}
		since then $u := v+w$ solves \eqref{873645836}. To this end define $A$ to be the negative Laplacian $-\Delta$ in $L_q(\Omega)$ with domain
		\begin{equation}
		D(A) := \{  w \in W^2_q(\Omega \backslash \Sigma) : w|_\Sigma = 0, \; n_{\partial\Omega} \cdot \nabla w|_{\partial\Omega} = 0   \}.
		\end{equation}
		Since $D(A)$ compactly embeds into $L_q(\Omega)$ by Sobolev embeddings, $A$ has compact resolvent and the spectrum $\sigma(A)$ only consists of eigenvalues of $A$ with finite multiplicity. We will show that zero is not a possible eigenvalue, hence $A$ is invertible. 
		
		Suppose $u \not = 0$ is a nontrivial eigenfunction to the eigenvalue $\lambda$. Since by well-known results the spectrum is independent of $q$, we may let $q=2$, cf. \cite{arendt}. Testing the resolvent equation with $u $ in $L_2(\Omega)$ and invoking the boundary condition yields
		\begin{equation}
		-\lambda |u|_{L_2(\Omega)}^2 = \int_{\Omega}^{} u \Delta u dx = - | \nabla u |^2_{L_2(\Omega)}.
		\end{equation}
		Whence if $\lambda = 0$, then $u \in D(A)$ has to be a constant function, hence zero since $u$ vanishes on $\Sigma$. This is a contradiction, hence $\lambda = 0 $ is not a possible eigenvalue. Therefore we may uniquely solve \eqref{873ddd645836} and the proof is complete.
			\end{proof}
			We conclude this section by the following observation, cf. \cite{pruessbuch}. Consider the special case where $(f,g_1,g_2) = (0,g,0)$. Define solution operators as follows. Let $T_0 g$ be the solution of the non-shifted problem \eqref{873645836} for $(f,g_1,g_2) = (0,g,0)$ and, for $\eta \geq \eta_0$, $T_\eta g$ the solution of \eqref{293876498765} with $(f,g_1,g_2) = (0,g,0)$. Then, $T_0 g - T_\eta g = \eta(\eta-\Delta_N)^{-1} T_0 g$. 
			Hereby, $z := (\eta- \Delta_N)^{-1} f$ solves the two-phase problem
			\begin{align}
			( \eta- \Delta)^{-1} z &= f, && \text{in } \Omega \backslash \Sigma, \\
			z|_\Sigma &= 0, && \text{on } \Sigma, \\
			(n_{\partial\Omega} | \nabla z) &= 0, && \text{on } \partial\Omega.
			\end{align}
				For details we refer to section 6.6 in \cite{pruessbuch}.
			
			\subsection{Cylindrical domains.} \label{384075647564574695} In the case where $n=3$ and $\Omega \subset \R^3$ is a bounded container, one needs a result for the elliptic model problem in the case where the top of the container meets the walls at a ninety degree angle. So let $G := \R_+ \times \R \times \R_+.$
			\begin{subequations} \label{93046875734rg}
				\begin{align}
				(\eta-\Delta) u &= f, && \text{in } G, \\
				\partial_1 u &= g_1, && \text{on } S_1 := \{ x_1 = 0, x_2 \in \R, x_3 \in \R_+\}, \\
				\partial_3 u &= g_2, &&\text{on } S_2 := \{ x_1 \in \R_+, x_2 \in \R, x_3 =0\}.
				\end{align}
			\end{subequations}
			The key observation is now that the two Neumann boundary conditions on $S_1$ and $S_2$ are compatible whenever $q < 2$. Suppose that we want to find a solution $u \in W^2_q(G)$ of the problem. Then by trace theory,
			\begin{equation}
			\nabla u |_{S_j} \in W^{1-1/q}_q(S_j), \quad j = 1,2.
			\end{equation}
			This yields necessary conditions for the data.
		 We see that on the set $\partial S_1 \cap \partial S_2 = \{ x_1 = x_3 = 0\}$ where the two boundary conditions meet, there is no compatibility condition for the data $g_1$ and $g_2$ in the system. This is due 
		 to the fact that since $q < 2$ the functions
		 $\nabla u|_{S_j}$ do not have a trace on $\partial S_j$. 
		  
		So let the given data satisfy 
		 \begin{equation}
		 g_j \in W^{1-1/q}_q(S_j),\quad  j=1,2.
		 \end{equation}
		 By a simple reflection we can recude the problem to a upper half-space problem with one Neumann condition and obtain full $W^2_q(G)$ regularity for the solution.
		 Let us state this observation in the following theorem.
		 \begin{theorem}
		 For all $(g_1,g_2) \in W^{1-1/q}_q(S_1) \times W^{1-1/q}_q(S_2)$ there exists a unique solution $u \in W^2_q(G)$ to problem \eqref{93046875734rg}. Furthermore, the solution map $[(g_1,g_2) \mapsto u]$ is continuous.
		 \end{theorem}
\section{The Neumann trace of the height function}
In this section we characterize the optimal trace space for the Neumann trace of the height function $h$ and show that it is a Banach algebra with respect to pointwise multiplication.
\begin{theorem} \label{8374650783645}
Let $n=2,3$, $0 < T \leq \infty, 5 \leq p < \infty$ and $q \in (5/3,2) \cap (2p/(p+1), 2p)$ and let $\Sigma$ be the flat interface $ \rnplus \cap \{ x_1 = 0 \}$. Let again $X_0 := W^{1-1/q}_q(\Sigma)$ and $X_1 := W^{4-1/q}_q(\Sigma)$. Then
\begin{align} \label{neumanntrace}
\prescript{}{0} W^1_p(0,&T;X_0) \cap L_p(0,T;X_1) \ni h \mapsto \\ &\mapsto \nabla h|_{\partial\Sigma} \in \prescript{}{0} F^{1 - 2/(3q)}_{pq}(0,T;L_q(\partial\Sigma)) \cap L_p(0,T;B^{ 3 -2/q}_{qq}(\partial\Sigma))
\end{align}
is bounded, linear, and has a continuous right inverse $E$, such that $\nabla Eg|_{\partial\Sigma} = g$ for all $g \in {_0 F}^{1 - 2/(3q)}_{pq}(0,T;L_q(\partial\Sigma)) \cap L_p(0,T;B^{ 3 -2/q}_{qq}(\partial\Sigma))  $.

In particular, there exists some constant $C>0$ independent of the length of the time interval $T$, such that
\begin{equation}
| \nabla h|_{\partial\Sigma} |_{ F^{1 - 2/(3q)}_{pq}(0,T;L_q(\partial\Sigma)) \cap L_p(0,T;B^{ 3 -2/q}_{qq}(\partial\Sigma))} \leq C |h|_{W^1_p(0,T;X_0) \cap L_p(0,T;X_1) },
\end{equation}
for all $h \in {_0 W}^1_p(0,T;X_0) \cap L_p(0,T;X_1) $ and
\begin{equation}
| E g|_{W^1_p(0,T;X_0) \cap L_p(0,T;X_1)} \leq C |g|_{ F^{1 - \frac{2}{3q}}_{pq}(0,T;L_q(\partial\Sigma)) \cap L_p(0,T;B^{ 3 -\frac{2}{q}}_{qq}(\partial\Sigma))}
\end{equation}
for all $g \in {_0 F}^{1 - 2/(3q)}_{pq}(0,T;L^q(\partial\Sigma)) \cap L_p(0,T;B^{ 3 -2/q}_{qq}(\partial\Sigma))$.
\end{theorem}

\begin{remark}
	The time trace at $t=0$ in $ {_0 F}^{1 - 2/(3q)}_{pq}(0,T;L_q(\partial\Sigma))$ is well defined since $1-2/(3q) > 1/p$ is ensured, cf. \cite{meyriesveraarpointwise}.
\end{remark}
\begin{proof} We may use Propositions 5.37 and 5.39 in \cite{kaipdiss} to get an embedding
	\begin{equation} \label{293864973dd}
	_0 W^1_p(0,T;X_0) \cap L_p(0,T;X_1) \into \prescript{}{0} F^{ 1-1/(3q) }_{pq}(0,T; W^1_q (\Sigma)),
	\end{equation}
	where the embedding constant is independent of $T$. This can be seen as follows. Since we restrict ourselves to functions with vanishing trace at $t=0$ we may extend the function to the half line $\R_+$ by reflection. We then apply the result in \cite{kaipdiss} and then restrict the extensions back to the finite interval $(0,T)$.
	
	Hence, \eqref{293864973dd} yields that for any $h \in {_0 W} ^1_p(0,T;X_0) \cap L_p(0,T;X_1),$
	\begin{equation}
	\nabla h \in  \prescript{}{0} F^{ 1-1/(3q)  }_{pq}(0,T; L_q (\Sigma)) \cap L_p(0,T;B^{3- 1/q}_{qq}(\Sigma)).
	\end{equation}
	Concering the traces of $\nabla h$ on the boundary $\partial\Sigma$, we use Proposition 5.23 in \cite{kaipdiss} to write this intersection space on the right hand side as an anisotropic Triebel-Lizorkin space $F^{s, \vec a}_{\vec p, q}$ and use the trace theory developed in \cite{kaipdiss} for these particular spaces. For a definition of $F^{s, \vec a}_{\vec p, q}$ we refer to Definition 5.15 in \cite{kaipdiss}. By Proposition 5.23 in \cite{kaipdiss},
	\begin{equation}
	F^{ 1-1/(3q)  }_{pq}(0,T; L_q (\Sigma)) \cap L_p(0,T;B^{3- 1/q}_{qq}(\Sigma)) \equiv F^{s, \vec a}_{\vec p, q}((0,T) \times \Sigma),
	\end{equation}
	where $s=1$,
	\begin{equation}
	\vec a = \left( \frac{1}{l}, ... , \frac{1}{l}, \frac{1}{t} \right), \quad  \vec p = (q,...,q,p) , \quad t = 1- \frac{1}{3q}, \quad l = 3- \frac{1}{q},
	\end{equation}
	where we take $n-1$ copies of $1/l$ and $q$, respectively. For taking now traces in these anisotropic Triebel-Lizorkin spaces we refer to \cite{johnsensickel2008}. With the notation used there in equations (2.1) and (2.11) we use Corollary 2.7 in \cite{johnsensickel2008} to get that the trace operator onto the boundary $\partial\Sigma$,
	\begin{equation}
	\operatorname{tr}_{\partial\Sigma} : F^{s, \vec a}_{\vec p, q}((0,T) \times \Sigma) \pfeil
	F^{s- \frac{1}{ql} , \vec {a''}}_{\vec {p''}, q}((0,T) \times \partial\Sigma) ,
	\end{equation}
	is bounded. Here $\vec{a''}$ and $\vec{p''}$ are used as introduced in the beginning of Section 2.1 in \cite{johnsensickel2008}. In our particular case,
	\begin{equation}
	\vec{a''} = \left( \frac{1}{l}, ... , \frac{1}{l}, \frac{1}{t} \right), \quad \vec{p''} = (q,...,q,p),
	\end{equation}
	taking now $n-2$ copies of $1/l$ and $q$, respectively. 
	We note at that point that by the order of integration with respect to the different exponents in $\vec p$ as explained in equation (3.1) in \cite{johnsensickel2008}, we have to take traces in "$x_1$-direction" in the notation of \cite{johnsensickel2008} and not in $x_n$-direction and therefore have to use Corollary 2.7 and not Corollary 2.8 in \cite{johnsensickel2008}.
	
	Again using Proposition 5.23 in \cite{kaipdiss},
	\begin{equation}
	F^{s- \frac{1}{ql} , \vec {a''}}_{\vec {p''}, q}((0,T) \times \partial\Sigma) = F^{ (s- \frac{1}{ql})t}_{pq}(0,T;L_q(\partial\Sigma)) \cap L_p(0,T; B^{ (s-\frac{1}{ql})l}_{qq}(\partial\Sigma)).
	\end{equation}
	Clearly,
	\begin{equation}
	\left(s- \frac{1}{ql}\right)t = \left( 1- \frac{1}{3q-1} \right)\left( 1- \frac{1}{3q} \right) = 1 - \frac{2}{3q},
	\end{equation}
	as well as
	\begin{equation}
	\left(s- \frac{1}{ql}\right)l = 3\left(s- \frac{1}{ql}\right)t = 3 - \frac{2}{q}.
	\end{equation}
	Hence
	\begin{equation} \label{eq18}
	F^{s- \frac{1}{ql} , \vec {a''}}_{\vec {p''}, q}((0,T) \times \partial\Sigma) = F^{ 1 - 2/(3q)  }_{pq}(0,T;L_q(\partial\Sigma)) \cap L_p(0,T; B^{ 3-2/q}_{qq}(\partial\Sigma)).
	\end{equation}
	Concludingly, we have shown so far that the mapping $h \mapsto \operatorname{tr}_{\partial\Sigma}\nabla h  $ between the spaces in \eqref{neumanntrace} is bounded.
	
	It remains to construct a continuous right inverse. This follows now by similar arguments using Corollary 2.7 in \cite{johnsensickel2008}. We omit the details here.
	
	We again point out that the constant is only independent of $T$ since we restrict ourselves to functions having vanishing trace at $t=0$.
\end{proof}
The boundedness of the trace operator can easily be generalized to the case of a curved interface by a standard argument involving a partition of unity and a localization argument.
\begin{theorem} \label{9486756}
	Let $\Omega \subseteq \Rn, n = 2,3$ bounded and smooth and $\Sigma$ a smooth interface of dimension $n-1$ in the sense that $\Sigma$ is a submanifold with interior inside $\Omega$ meeting the boundary at a ninety degree angle. Then
	$\operatorname{tr}_{\partial\Sigma} \nabla_\Sigma : {_0 X_T} \pfeil F^{1 - 2/(3q)}_{pq}(0,T;L_q(\partial\Sigma)) \cap L_p(0,T;B^{ 3 -2/q}_{qq}(\partial\Sigma))$ is bounded. 
\end{theorem}

The next result states that the Neumann trace space is a Banach algebra under pointwise multiplication.
\begin{theorem} \label{banachalgebra}
	Let $n=2,3$, $0 < T \leq +\infty, 3 \leq p < \infty$ and $q \in (3/2,2) \cap (2p/(p+1) , 2p)$. Then the Neumann trace space with vanishing time trace at $t=0$ above is a Banach algebra, that is, the product estimate
	\begin{align} \label{9838rt}
	\| fg \|&_{F^{1 - \frac{2}{3q}}_{pq}(0,T;L^q(\partial\Sigma)) \cap L^p(0,T;B^{ 3 -\frac{2}{q}}_{qq}(\partial\Sigma))} \lesssim \\ &\lesssim \| f \|_{F^{1 - \frac{2}{3q}}_{pq}(0,T;L^q(\partial\Sigma)) \cap L^p(0,T;B^{ 3 -\frac{2}{q}}_{qq}(\partial\Sigma))}\| g \|_{F^{1 - \frac{2}{3q}}_{pq}(0,T;L^q(\partial\Sigma)) \cap L^p(0,T;B^{ 3 -\frac{2}{q}}_{qq}(\partial\Sigma))} \nonumber
	\end{align}
	holds for all $f,g \in \prescript{}{0} F^{1 - 2/(3q)}_{pq}(0,T;L_q(\partial\Sigma)) \cap L_p(0,T;B^{ 3 -2/q}_{qq}(\partial\Sigma)).$ In particular, the constant in \eqref{9838rt} is independent of the length of the time interval.
\end{theorem}
\begin{proof}
We begin by showing that
	\begin{equation} \label{embemb111}
	\prescript{}{0} F^{1 - 2/(3q)}_{pq}(0,T;L_q(\partial\Sigma)) \cap L_p(0,T;B^{ 3 -2/q}_{qq}(\partial\Sigma)) \into L_\infty(0,T;L_\infty(\partial\Sigma)).
	\end{equation}
	Using Proposition 5.38 in \cite{kaipdiss}, the space on the left hand side continuously embeds into
	\begin{equation}
	 \prescript{}{0} H^{(1-\frac{2}{3q})\theta}_p(0,T;B^{(3-2/q)(1-\theta)}_{qq}(\partial\Sigma))
	\end{equation}
	for any $\theta \in (0,1)$ where the embedding constant is independent of $T$. 
	Note that if $\theta$ is so small such that the space on the right hand side does not have a well defined time trace at $t=0$, we simply replace it with $ H^{(1-\frac{2}{3q})\theta}_p(0,T;B^{(3-2/q)(1-\theta)}_{qq}(\partial\Sigma))$.
	
	Now,
	since $n=2$ or $3$, the boundary $\partial\Sigma$ has at most dimension $1$, whence the latter space on the right hand side surely embeds into $L_\infty(0,T;L_\infty(\partial\Sigma))$, if
	\begin{equation}
	\left(1-2/(3q)\right)\theta - 1/p > 0, \quad \left(3-2/q \right)(1-\theta) - 1/q> 0.
	\end{equation}
	These both equations are equivalent to finding some $\theta \in (0,1)$ such that
	\begin{equation}
	\frac{1}{p} \frac{3q}{3q-2} < \theta < 1-\frac{1}{3q-2}.
	\end{equation}
	Simple calculations show that for any $q \in (3/2,2),$
	\begin{equation}
	1- \frac{1}{3q-2} > \frac{3}{5}, \quad \frac{3q}{3q-2} < \frac{9}{5},
	\end{equation}
	whence $p \geq 3$ ensures $\theta = 3/5$ is a solid choice. Therefore we know for sure that in any of our cases the Neumann trace space embeds continuously into $L_\infty(0,T;L_\infty(\partial\Sigma)).$
	
	Using well-known paraproduct estimates, cf. \cite{danchinbuch},
	\begin{align}
	|&fg |_{L_p(0,T;B^{3-2/q}_{qq}(\partial\Sigma))} \leq\\
	&\leq \left|  |f(t)|_{L_\infty} |g(t)|_{ B^{3-2/q}_{qq} }  \right|_{L_p(0,T)} + \left|  |f(t)|_{ B^{3-2/q}_{qq} } |g(t)|_{L_\infty}  \right|_{L_p(0,T)} \\
	&\leq | f |_{L_\infty(L_\infty)} | g |_{L_p(B^{3-2/q}_{qq})} +  | f |_{L_p(B^{3-2/q}_{qq})}| g |_{L_\infty(L_\infty)}.
	\end{align}
	From Proposition 5.7 in \cite{meyriesveraarpointwise} we get
	\begin{align}
	| fg |_{ F^{1 - \frac{2}{3q}}_{pq}(0,T;L_q(\partial\Sigma)) } &\lesssim | f|_{ F^{1 - \frac{2}{3q}}_{pq}(0,T;L_q(\partial\Sigma)) } | g |_{L_\infty(0,T;L_\infty(\partial\Sigma))} \\
	&+  | f |_{L_\infty(0,T;L_\infty(\partial\Sigma))} | g|_{ F^{1 - \frac{2}{3q}}_{pq}(0,T;L_q(\partial\Sigma)) }.
	\end{align}
	These two estimates and \eqref{embemb111} finish the proof.
\end{proof}

\paragraph{Acknowledgements.} M.R. would like to thank Harald Garcke for pointing out the results in \cite{depnerdiss} and for inspiring discussions concerning the proof of the stability result. The work of M.R. is financially supported by the DFG graduate school GRK 1692. The support is gratefully acknowledged.

\end{appendices}
\bibliography{bibo}

\begin{thebibliography}{10}

\bibitem{abelsbuch}
H.~Abels.
\newblock {\em Pseudodifferential and Singular Integral Operators}.
\newblock De Gruyter, 2011.

\bibitem{abelswilke}
H.~Abels and M.~Wilke.
\newblock Well-posedness and qualitative behaviour of solutions for a two-phase
  {N}avier-{S}tokes-{M}ullins-{S}ekerka system.
\newblock {\em Interfaces and Free Boundaries}, 15:39--75, 2013.

\bibitem{amann}
H.~Amann.
\newblock Nonhomeongeoeus linear and quasilinear elliptic and parabolic
  boundary value problems.
\newblock {\em Function Spaces, Differential Operators and Nonlinear Analysis},
  pages 9--126, 1993.

\bibitem{amannlineartheory}
H.~Amann.
\newblock {\em Linear and Quasilinear Parabolic problems, {V}olume I: Abstract
  linear theory}.
\newblock Birkhaeuser, 1995.

\bibitem{arendt}
W.~Arendt.
\newblock {G}aussian estimates and interpolation of the spectrum in ${L}^p$.
\newblock {\em Differential and integral equations}, 7(5):1153--1168, 1994.

\bibitem{danchinbuch}
H.~Bahouri, J.~Chemin, and R.~Danchin.
\newblock {\em {F}ourier {A}nalysis and nonlinear partial differential
  equations}.
\newblock Springer, 2011.

\bibitem{bourgain1984}
J.~Bourgain.
\newblock Extension of a result of {B}enedek, {C}alderón and {P}anzone.
\newblock {\em Ark. Mat.}, 22(1-2):91--95, 12 1984.

\bibitem{depnerdiss}
D.~Depner.
\newblock {S}tability analysis of geometric evolution equations with triple
  lines and boundary contact.
\newblock {\em Dissertation, Universität Regensburg}, 2010.

\bibitem{dore2000}
G.~Dore.
\newblock Maximal regularity in ${L}^p$ spaces for an abstract {C}auchy
  problem.
\newblock {\em Adv. Differential Equations}, 5(1-3):293--322, 2000.

\bibitem{engelnagel}
K.~Engel and R.~Nagel.
\newblock {\em {O}ne-parameter semigroups for linear evolution equations}.
\newblock Springer, New York, 2000.

\bibitem{eschersimonett}
J.~Escher and G.~Simonett.
\newblock A {C}enter {M}anifold {A}nalysis for the {M}ullins-{S}ekerka {M}odel.
\newblock {\em Journal of Differential Equations}, 143:267--292, 1998.

\bibitem{johnsendiffeo}
J.~Johnsen, S.~Munch~Hansen, and W.~Sickel.
\newblock Anisotropic, {M}ixed-{N}orm {L}izorkin-{T}riebel {S}paces and
  diffeomorphic maps.
\newblock {\em Journal of Function Spaces}, 2014.

\bibitem{johnsensickel2008}
J.~Johnsen and W.~Sickel.
\newblock On the trace problem for {L}izorkin-{T}riebel spaces with mixed
  norms.
\newblock {\em Mathematische Nachrichten}, 281(5):669--696, 2008.

\bibitem{kaipdiss}
M.~Kaip.
\newblock {\em General parabolic mixed order systems in $L^p$ and
  applications}.
\newblock PhD thesis, Universität Konstanz, 2012.

\bibitem{kaipsaal}
M.~Kaip and J.~Saal.
\newblock The permanence of {R}-boundedness and property alpha under
  interpolation and applications to parabolic systems.
\newblock {\em J. Math. Sci. Univ. Tokyo}, 19(3):359--407, 2012.

\bibitem{koehnepruesswilke}
M.~Köhne, J.~Prüss, and M.~Wilke.
\newblock Qualitative behaviour of solutions for the two-phase
  {N}avier-{S}tokes equations with surface tension.
\newblock {\em Mathematische Annalen}, 356(2):737--792, 2013.

\bibitem{lunardioptimal}
A.~Lunardi.
\newblock {\em Analytic semigroups and {O}ptimal {R}egularity in {P}arabolic
  {P}roblems}.
\newblock Springer, 1995.

\bibitem{lunardiinterpol}
A.~Lunardi.
\newblock {\em Interpolation {T}heory}.
\newblock Springer, 2009.

\bibitem{meyriesveraartraces}
M.~Meyries and M.~Veraar.
\newblock Traces and {E}mbeddings of anisotropic function spaces.
\newblock {\em Mathematische Annalen}, 360:571--606, 2014.

\bibitem{meyriesveraarpointwise}
M.~Meyries and M.~Veraar.
\newblock Pointwise multiplication on vector-valued function spaces with power
  weights.
\newblock {\em Journal of Fourier Analysis and Applications}, 1:95--136,
  February 2015.

\bibitem{REFPAPERPRUESS}
J.~Pr\"uss.
\newblock Maximal {R}egularity for abstract parabolic problems with
  inhomogeneous boundary data in ${L}_p$-spaces.
\newblock {\em Mathematica Bohemica}, 127(2):311--327, 2002.

\bibitem{pruessbuch}
J.~Pr\"uss and G.~Simonett.
\newblock {\em Moving interfaces and quasilinear parabolic evolution
  equations}.
\newblock Birkhäuser Verlag, 2016.

\bibitem{pruessmulsek}
J.~Prüss, G.~Simonett, and R.~Zacher.
\newblock On convergence of solutions to equilibria for quasilinear parabolic
  problems.
\newblock {\em Journal of Differential Equations}, 246(10):3902--3931, 2009.

\bibitem{runst}
T.~Runst.
\newblock Mapping properties of nonlinear operators in spaces of
  {T}riebel-{L}izorkin and {B}esov type.
\newblock {\em Analysis Mathematica}, 12(4):313--346, 1986.

\bibitem{triebel}
H.~Triebel.
\newblock {\em Theory of Function Spaces}.
\newblock Birkhauser, 2000.

\bibitem{vogel}
T.~Vogel.
\newblock Sufficient conditions for capillary surfaces to be energy minima.
\newblock {\em Pacific Journal of Mathematics}, 194(2):469--489, 2000.

\end{thebibliography}
\end{document}